\documentclass[11pt,a4paper]{amsart}

\usepackage{amsfonts}
\usepackage{amssymb}
\usepackage{amsmath}
\usepackage{amsthm}
\usepackage{amssymb}
\usepackage[english]{babel} 
\usepackage{bm}
\usepackage{color}
\usepackage{hyperref}
\usepackage{enumitem}
\usepackage{graphicx}
\usepackage{hyperref} 
\usepackage[utf8]{inputenc} 
\usepackage{lscape}
\usepackage{multicol}
\usepackage{multirow}
\usepackage{soul}
\usepackage{url}
\usepackage{xcolor}
\usepackage[normalem]{ulem}
\usepackage{enumitem}
\usepackage{soul}

\usepackage{listings}

	\graphicspath{{images/}} 

\newcommand{\conv}  {\operatorname{conv} }
\newcommand{\width}  {\operatorname{width} }
\newcommand{\length}{\operatorname{length}}
\renewcommand{\int}  {\operatorname{int} }
\newcommand{\vol}  {\operatorname{vol} }
\newcommand{\zono}{\operatorname Z}
\newcommand{\fractional}[1]  { \{#1\} }
\newcommand{\double}[1]  { {#1}\pm #1 }
\newcommand{\val}  {\operatorname{val} }
\newcommand{\dep}  {\operatorname{dep} }

\newcommand{\gap}  {\gamma } 
\newcommand{\sgap}  {\gap^{\operatorname{sh}}} 
\newcommand{\covering}  {\mu} 
\newcommand{\cminimum}  {\kappa} 
\newcommand{\mucosimple}{\mu^{\text{c}}}
\newcommand{\musLR}{\mu^{\text{s}}}

\newtheorem{theorem}{Theorem}[section]
\newtheorem{corollary}[theorem]{Corollary}
\newtheorem{lemma}[theorem]{Lemma}
\newtheorem{proposition}[theorem]{Proposition}

\newtheorem{example}[theorem]{Example}
\newtheorem{remark}[theorem]{Remark}

\newtheorem{definition}[theorem]{Definition}
\newtheorem{conjecture}[theorem]{Conjecture}
\newtheorem{question}[theorem]{Question}

\newcommand{\bb}{\mathbf{b}}
\newcommand{\cc}{\mathbf{c}}
\newcommand{\ee}{\mathbf{e}}

\newcommand{\oo}{\mathbf{0}}
\newcommand{\pp}{\mathbf{p}}
\newcommand{\qq}{\mathbf{q}}

\newcommand{\uu}{\mathbf{u}}
\newcommand{\vv}{\mathbf{v}}
\newcommand{\kk}{\mathbf{k}}
\newcommand{\ww}{\mathbf{w}}
\newcommand{\xx}{\mathbf{x}}
\renewcommand{\ss}{\mathbf{s}}
\newcommand{\yy}{\mathbf{y}}
\newcommand{\zz}{\mathbf{z}}

\newcommand{\ZZ}{\mathbb{Z}}

\newcommand{\RR}{\mathbb{R}}

\DeclareMathOperator{\dist}{dist}

\newcommand{\R}{\mathbb R}

\newcommand{\Z}{\mathbb Z}

\usepackage[colorinlistoftodos, textwidth=3cm, textsize=small]{todonotes}

\usepackage{booktabs}

\begin{document}

\title[Coloopless zonotopes and the Lonely Runner Conjecture]{Coloopless zonotopes and counterexamples to the Shifted Lonely Runner Conjecture}
\hypersetup{
    pdftitle={Coloopless zonotopes and counterexamples to the Shifted Lonely Runner Conjecture},
    pdfauthor={Monica Blanco, Francisco Criado, Francisco Santos},
    pdfkeywords={zonotopes, covering radius, Lonely Runner Conjecture, lattice polytopes}
}

    \thanks{%
        Supported by grant PID2022-137283NB-C21
        funded by MICIU/AEI/10.13039/\allowbreak 501100011033 and FEDER, UE.
        The authors  thank two anonymous referees for several remarks, corrections, and suggestions which have improved the paper.
    }%
    \thanks{$^*$Corresponding author.}%

\author[M.~Blanco]{M\'onica Blanco}
\author[F.~Criado]{Francisco Criado\textsuperscript{*}}
\author[F.~Santos]{Francisco Santos}

\address[M.~Blanco, F.~Santos]{
    Departamento de Matem\'aticas, Estad\'istica y Computa\-ci\'on\\
    Universidad de Cantabria\\
         Santander\\
         Spain}
\email{monica.blancogomez@unican.es,
francisco.santos@unican.es}

\address[F.~Criado]{Departamento de Matem\'aticas \\
    CUNEF Universidad\\
        Madrid \\ 
        Spain}
\email{francisco.criado@cunef.edu}

\begin{abstract}
Henze and Malikiosis (2017) have shown that the Lonely Runner Conjecture (LRC) can be restated as a convex-geometric question on the so-called \emph{LR zonotopes}, lattice zonotopes with one more generator than their dimension. This relation naturally suggests a more general statement,  the \emph{shifted LRC},  the zonotopal version of which concerns a classical parameter, the covering radius. 

In this paper we do two things:

\begin{enumerate}
\item We show explicit counterexamples to both the shifted Lonely Runner Conjecture (starting at $n=5$) and to the Lonely Vector Property of Malikiosis-Schymura-Santos (2025).

\item We push the analogies between the two versions of LRC and their zonotopal counterparts, in particular highlighting that the proofs of the finite-checking Theorems A and B in Malikiosis-Schymura-Santos (2025) are more transparent, and the statements more general, if regarded in terms of two quite general classes of lattice zonotopes: the coloopless zonotopes that we introduce here and the cosimple ones, already defined by them. 
\end{enumerate}
\end{abstract}
\maketitle
\setcounter{tocdepth}{1}
\tableofcontents

\section{Introduction}

\subsection*{The Lonely Runner Conjectures}

The Lonely Runner Conjecture is part (i) of the following statement, posed by Wills in 1968~\cite{willslrc} and implicit in earlier work of his~\cite{wills65,willsthesis,wills67}.
 The shifted version in part (ii) first appeared in print as~\cite[Conjecture 10]{sLRC} (2019), attributed to a personal communication of Wills as well. We state both as ``conjectures'', but this paper contains explicit counterexamples to part (ii) starting with $n=5$.

\begin{conjecture}
\label{conj:conj}
Let $v_1,\dots,v_n$ be non-zero real numbers. Then:
\begin{enumerate}[label=\upshape(\roman*),ref=\thetheorem .\roman*]
\item 
\label{conj:LRC}
(Lonely Runner Conjecture, LRC). There is a $t\in \R$ such that
\[
\dist(v_i t, \Z) \ge \frac1{n+1}, \qquad \forall i\in \{1,\dots,n\}.
\] 

\item 
\label{conj:sLRC}
(Shifted Lonely Runner Conjecture, sLRC). If the $|v_i|$ are all different, for every $s_1,\dots,s_n\in \R$
there is a  $t\in\R$ such that
\[
\dist(s_i + v_i t, \Z) \ge \frac1{n+1}, \qquad \forall i\in \{1,\dots,n\}.
\] 
\end{enumerate}
\end{conjecture}

In version (ii) if the $v_i$ are not assumed different then $v_i=\dots v_n$ and $s_i = \tfrac{i}{n}$ for all $i$ is a trivial counterexample.
In (i), in contrast,  an instance with repeated $v_i$ follows from the same instance with repetitions deleted. \medskip

One reason why the LRC has attracted attention is that it can be interpreted in various fashions in terms of \emph{view obstruction}, \emph{billiard trajectories} in a cubical pool, or rectilinear trajectories in the torus $\RR^n/\ZZ^n$, among others. See~\cite{NamingLRC, cusickviewob, Schoenberg76} or the recent survey~\cite{perarnauserra2024thelonely}. 
The name ``lonely runner'' comes from interpreting the $v_i$ as velocities of $n$ people running along a closed track of length one. In version (i) all runners start at the origin and in version (ii) each runner starts at an initial position $s_i$. 
In both cases the conjecture is that there is a time $t$ when an ``$(n+1)$-th runner'', who stayed at the origin, is ``lonely'', meaning that it is at distance at least $\frac1{n+1}$ from the rest.

\begin{remark}
Because of the interpretation via an extra runner, in the literature the conjectures as we stated them are considered the case of ``$n+1$ runners''. E.g., \cite{7lrc, Rosenfeld1, Rosenfeld2, Trakulthongchai, SungkawichaiTrakulthongchai} solve the cases $n=6,\dots,12$ of the LRC but their titles speak of seven to thirteen runners.
The shifted version is solved up to $n=4$, that is, five runners~\cite{ACS4slrz}.
\end{remark}

To analyze the conjectures it is customary to introduce the following parameters, although notation varies from one paper to another:

\begin{definition}
\label{def:loneliness}
Let $\vv=(v_1,\dots, v_n)\in (\RR\setminus \{0\})^n$ be a velocity vector. 
\begin{enumerate}[label=\upshape(\roman*),ref=\thetheorem .\roman*]
\item 
\label{def:loneliness-LR}
The (unshifted) \emph{loneliness gap} of $\vv$ is 
\[
\gap(\vv) := \sup_{t\in \RR} \min_{i\in [n]}\dist(v_i t, \Z).
\]
\item The \emph{loneliness gap} of $\vv$ \emph{shifted by} an $\ss =(s_1,\dots,s_n) \in [0,1]^n$ is 
\label{def:loneliness-sLR}
\[
\gap(\vv; \ss):= \sup_{t\in \RR} \min_{i\in [n]}\dist(s_i + v_i t, \Z),
\]
and the \emph{shifted loneliness gap} of $\vv$ is $\sgap(\vv):= \min_{\ss\in [0,1]^n}\,\gap(\vv; \ss)$.
\end{enumerate}
\end{definition}

In this language Conjectures~\ref{conj:LRC} and~\ref{conj:sLRC} say, respectively, that $\gap(\vv)$ and $\sgap(\vv)$ are at least 
$1/(n+1)$, the former for all $\vv$ with no zero entries and the latter for all $\vv$ with no zero or repeated (in absolute value) entries. 
It is well-known that in both cases there is no loss of generality in assuming that $\vv$ is an integer vector, and that it is \emph{primitive}; that is, $\gcd(\vv)=1$.

\subsection*{Lonely runner zonotopes}

All the results and proofs in this paper use a restatement of Conjecture~\ref{conj:conj} in terms of certain lattice zonotopes, the so-called \emph{lonely runner zonotopes}. This connection was first explored by Henze and Malikiosis~\cite{zonorunners}.
Recall that a \emph{zonotope} is a Minkowski sum of segments; equivalently, an affine image of a cube. More specifically, 
the \emph{zonotope generated by}
 a finite set of vectors  $U=\{\uu_1,\dots, \uu_n\}\subset \R^d$ is
\[
\zono[U]:=\sum_{i=1}^{n} [\oo, \uu_i] = \Big\{\sum_{i=1}^{n}\lambda_i \uu_i : \lambda_i \in [0,1] \ \forall i\in\{1,\dots,n\}\Big\}.
\]
Every zonotope can be put in the form $Z[U]$ by translating it to have a vertex at the origin. 
\emph{Lattice zonotopes} are zonotopes with $U\subset\Z^d$. 

For each subset $S\subset U$ of size $d$, the number $|\det(S)|$ equals the volume of the  \emph{parallelepiped}  $\zono[S]$.
The \emph{volume vector} (or \emph{Pl\"ucker vector}) of $\zono[U]$ (or of $U)$ is the vector in $\R^{\binom{n}{d}}$ consisting of these numbers. It is well-known that the volume of $\zono[U]$ equals the sum of the entries in its volume vector~(see, e.g., \cite[Lemma 9.1]{BeckRobins}).

\begin{definition}[\cite{Malikiosis2024LinExpCheckingLRC, ACS4slrz}]
\label{def:LRZ}
\begin{enumerate}[label=\upshape(\roman*),ref=\thetheorem .\roman*]
\item A \emph{lonely runner zonotope} (or LR zonotope, for short) of dimension $d$ is any $d$-dimensional lattice zonotope with $d+1$ generators whose volume vector has no zero entries. 
\item The LR zonotope is called \emph{strong} (sLR zonotope) if its volume vector entries are all different.
\end{enumerate}
\end{definition}

LR zonotopes exist for every volume vector $\vv\in \Z_{>0}^n$ and,
when $\vv$ is primitive, the LR zonotope with volume vector $\vv$ is unique modulo unimodular isomorphism; we denote it $\zono(\vv)$ (see \cite{Malikiosis2024LinExpCheckingLRC}, or \cite[Proposition 2.2]{ACS4slrz}). 
Observe that if $\sum_{i=1}^n \lambda_i \uu_i =0$ is the primitive linear dependence among the generators $U=\{\uu_1,\dots,\uu_n\}\subset \Z^{n-1}$ then the volume vector of $\zono[U]$ is $(|\lambda_i|)_{i=1}^n$.

\begin{remark}
$\zono(\vv)$ can be also described as the image of the cube $[0,1]^n$ under the affine projection with kernel $\R\vv$.
Indeed, since $\Z^n/\Z\vv$  is a full-dimensional lattice in $\R^{n}/\R\vv$, there is a linear isomorphism $\phi: \R^{n}/\R\vv \stackrel{\cong}{\longrightarrow} \R^{n-1}$ with $\phi(\Z^n/\Z\vv) = \Z^{n-1}$.  The composition of the quotient map $\R^n \to\R^{n}/\R\vv$ with  any such isomorphism $\phi$ gives a linear projection $\pi_\vv: \RR^n \to\R^{n-1}$ with $\pi_\vv(\ZZ^n) = \Z^{n-1}$ and $\ker(\pi_\vv)= \R\vv$. We have:

\begin{proposition}
\label{prop:isomorphism}
Let  $\pi_\vv$ be any such projection.
If $\vv$ is primitive then $\pi_\vv([0,1]^n)$ is a lattice $(n-1)$-zonotope with volume vector $\vv$.
\end{proposition}

\end{remark}

\begin{proof}
Observe that  the zonotope $Z: = \pi_\vv([0,1]^n)$ has generators $\uu_i := \pi_\vv(\ee_i) \in \pi_\vv(\Z^n) = \Z^{n-1}$, ($i=1,\dots,n$). Hence, it is a lattice zonotope and its volume vector, which we call $\vv'$, is integral. Moreover,  we have that  $0 = \sum_i v'_i \uu_i = \sum_i v'_i \pi_\vv(\ee_i) = \pi_\vv(\vv')$, so $\vv' \in \ker(\pi_\vv) = \R\vv$. 
Since the $\uu_i$ generate $\pi_\vv(\Z^n)$, which equals $\Z^{n-1}$, $\vv'$ is primitive. 
(We are using here the fact that the volume vector of a lattice zonotope is primitive if and only if its generators span the whole lattice $\Z^d$).
Hence, $\vv= \pm \vv'$.
\end{proof}

In what follows $\covering(K)$ and $\cminimum(K)$ denote the \emph{covering radius} and the \emph{first $\cc$-minimum} of $K$, for any convex body $K$. The former is a classical convex-geometric parameter and the latter is formally introduced in this paper, but is implicit in previous work. See details in Section~\ref{sec:prelim}.

\begin{proposition}[Henze-Malikiosis~{\cite[Theorem 1.1 and Corollary 5.1]{zonorunners}}]
\label{prop:zonorunners}
\label{prop:gamma-kappa-mu}
For each primitive vector $\vv\in \Z^{n}$
\begin{align}
\label{eq:gamma-kappa-mu}
\gap(\vv) =\tfrac12 - \tfrac12 \cminimum(\zono(\vv)), 
\qquad
\sgap(\vv) =\tfrac12 - \tfrac12 \covering(\zono(\vv)).
\end{align}
\end{proposition}

Hence, the LRC (resp., sLRC) Conjecture is equivalent to the following statement:  the first $\cc$-minima (resp.,  the covering radii) of all LR zonotopes (resp., sLR zonotopes) of a given dimension $d=n-1$ are upper-bounded by $\tfrac{d}{d+2}$. This constitutes the starting point of this paper. \medskip

We now describe the contents of Sections~\ref{sec:counter},~\ref{sec:LRC} and~\ref{sec:sLRC}, in this order. Section~\ref{sec:prelim} contains additional preliminaries on convex geometry and zonotopes.

\subsection*{Counter-examples to sLRC}
The main result in Section~\ref{sec:counter}, and perhaps in the paper, is:

\begin{theorem}[Propositions~\ref{prop:counters5}--\ref{prop:countersn}]
\label{thm:12345}
\begin{align*}
\sgap(1,2,3,4,5) &\ = \frac{15}{94} < \frac16. \\
\sgap(2,3,4,5,6,8) &\ =\frac{2}{15} \ <\  
\sgap(1,2,3,4,5,6)\ =\ \frac{9}{67}\ <\ \frac17. \\
\sgap(1,2, \dots, n) &\ <\ \frac1{n+1} \qquad \forall n\in\{7,...,17\}.
\end{align*}

Hence, the shifted LRC (Conjecture~\ref{conj:sLRC}) is false for $n\in\{5,...,17\}$.
\end{theorem}

These values and bounds were found via Proposition~\ref{prop:zonorunners} with a novel algorithm to compute $\covering(\zono(\vv))$. The algorithm shares ideas 
with the one we used in \cite{ACS4slrz}  for covering radii of arbitrary rational polytopes, but it takes advantage of particular features of LR zonotopes. Most notably, scaling coordinates in $\R^n$ by the entries of the vector $\vv$ we regard LR zonotopes as \emph{polytropes} and introduce tropical geometry ideas to reduce the impact of numerical issues by orders of magnitude, resulting in a much better practical performance.  
See Section~\ref{sec:counter} for more details, both on the algorithm and on our counterexamples and their statistics.%
\footnote{Section~\ref{sec:counter} can be read skipping the rest of this section and Sections~\ref{sec:prelim}, \ref{sec:LRC} and~\ref{sec:sLRC}.}
Our source code is available in the repository \cite{code_git}, where the more computationally inclined reader can also find a script \lstinline|scripts/verify_solution.py| which certifies a $\sgap$ with exact rational arithmetic and is optimized for readability.

Observe, however, that once a counter-example is found, it can be verified with little computation. E.g., to prove that $\sgap(1,2,3,4,5) \leq \frac{15}{94}$ one simply needs to check the correctness of Figure~\ref{fig:12345}.

\begin{figure}[htb]
\includegraphics[scale=0.17]{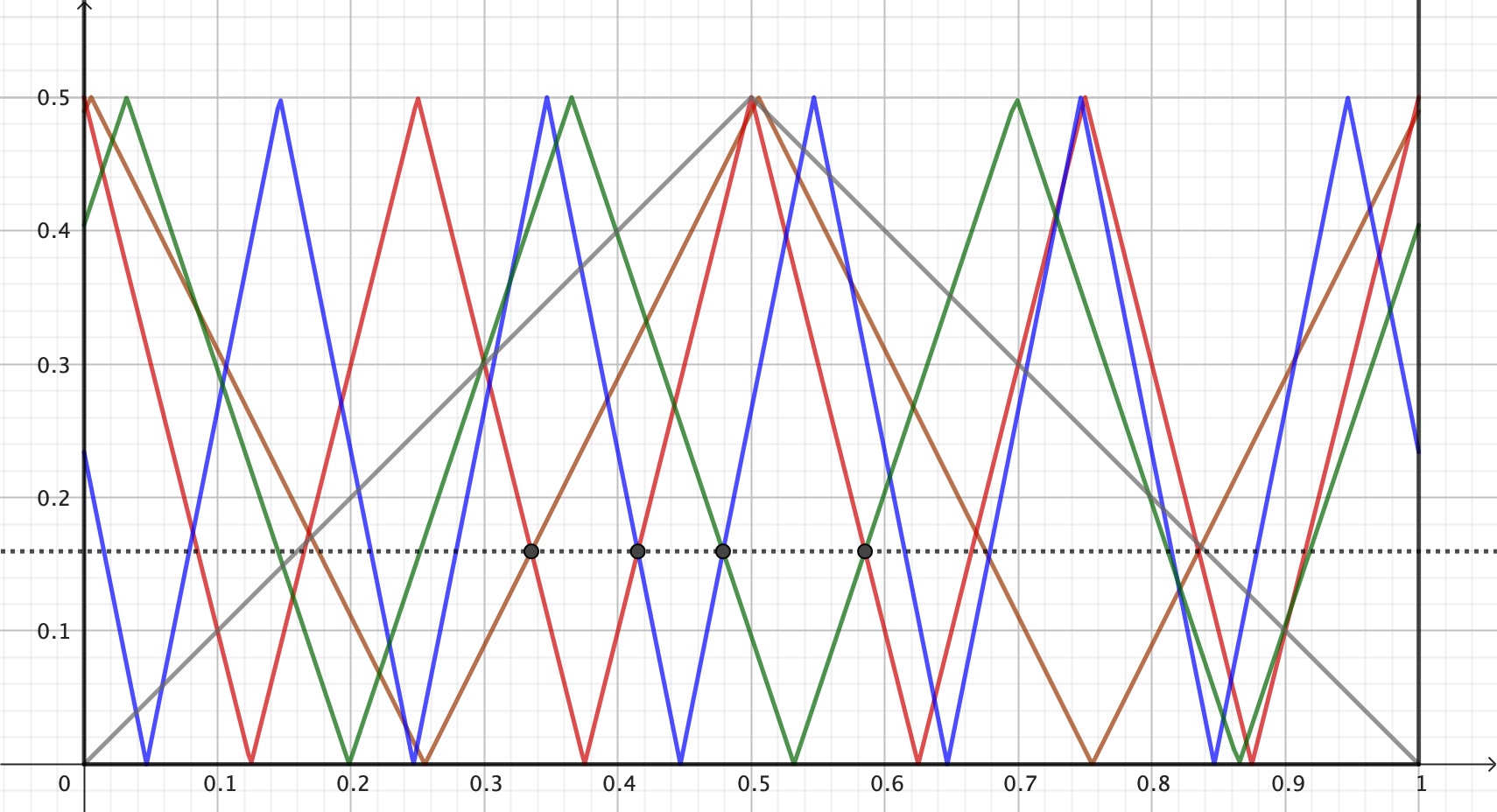}
\caption{Distance to the origin in terms of $t$ of five runners with $\vv=(1,2,3,4,5)$ and $\ss=\tfrac1{94}(0,46,38,47,72)$. The fact that for every $t\in [0,1]$ there is some runner on or below the horizontal line $\{y=\tfrac{15}{94}\}$ shows that $\sgap(1,2,3,4,5) \leq \tfrac{15}{94}$. The four dots along this line are the instants when the minimum distance from the runners to the origin \emph{equals} $\tfrac{15}{94}$.
}
\label{fig:12345}
\end{figure}

\subsection*{Coloopless zonotopes and the Lonely Runner Conjecture}

Malikiosis et at.~\cite{Malikiosis2024LinExpCheckingLRC} showed that  in order to verify Conjecture~\ref{conj:LRC} for a given $n$ it suffices to check finitely many velocity vectors.
In the following statement, for a  vector $\vv\in \Z^n$ and each subset $S\subset [n]$ we define
\[ 
\vv_S:=\gcd(|v_i|:i\in S). 
\]

\begin{theorem}[{\cite[Theorem A]{Malikiosis2024LinExpCheckingLRC}}] \ 
\label{thm:finite-LR}
If Conjecture~\ref{conj:LRC} holds for all velocity vectors of length $n-1$ and for integer velocity vectors of length $n$ satisfying 
$\sum_{S\subseteq [n]} \vv_S \le \binom{n+1}2^{n-1}$ then it holds for all velocity vectors of length $n$.
\end{theorem}

The quantity $\sum_{S\subseteq [n]} \vv_S$ that appears in this statement equals the number of lattice points in the corresponding LR zonotope, and is larger than the  volume of it (\cite[Corollary 2.3]{Malikiosis2024LinExpCheckingLRC}):
\begin{equation}
\label{eq:lrc-points}
\displaystyle  |\zono(\vv) \cap \ZZ^{n-1}| = \sum_{S\subseteq [n]}  \vv_S
> \sum_{i\in [n]}  v_{i} = \vol(\zono(\vv)).
\end{equation}

\begin{remark}
The first  finite-checking result for LRC was proved by Tao~\cite{Tao}, but with a much worse and much less explicit bound. A bound that is close to that of Theorem~\ref{thm:finite-LR}, but still worse, is in \cite{girikravitz2023structurelonelyrunnerspectra}.

In the other direction, Rosenfeld~\cite{Rosenfeld1} has found a way to prove lower bounds on the velocities of possible counter-examples to LRC: he stated conditions guaranteeing, for a given prime number $p$, that any velocity vector with $\prod_{i\in [n]} v_i$ not a multiple of $p$ satisfies Conjecture~\ref{conj:LRC}. Proving those conditions for a set of primes whose product exceeds the $n$-th power of the bound in Theorem~\ref{thm:finite-LR} establishes the Lonely Runner Conjecture. This approach has been successfully implemented for $n\in[7,12]$  (\cite{Rosenfeld1, Rosenfeld2, Trakulthongchai,SungkawichaiTrakulthongchai}).
\end{remark}

Our main contribution for the unshifted LRC, developed in Section~\ref{sec:LRC}, is a reworking of the proof of Theorem~\ref{thm:finite-LR} which clarifies (we believe) the original proof but also shows that  lonely runner zonotopes are part of a much larger class of zonotopes which we call \emph{coloopless zonotopes}, on which the parameter $\cminimum$ has the same exact maximum.
Coloopless zonotopes are those having generators with the following equivalent properties:

\begin{proposition}[Propositions~\ref{prop:co},~\ref{prop:Gale}, and~\ref{prop:cowidth}]
\label{prop:coloopless}
Let  $U=\{\uu_1,\dots,\uu_n\}\subset \ZZ^d$ be integer vectors.
Then, the following conditions are equivalent:
\begin{enumerate}
\item There is a linear dependence $\sum_{i=1}^n \lambda_i \uu_i =0$ with $\lambda_i\ne 0$ for all $i$ (i.e., $U$ is coloopless, Definition~\ref{def:co}).
\item No linear hyperplane contains all but one of the $\uu_i$ (``$U$ has no coloop'').
\item The Gale transform $U^*$ of $U$ contains no zero vector.
\item $\zono[U]$ has width at least $2$ with respect to the lattice generated by $U$.
\end{enumerate}
\end{proposition}

Any of the first two conditions show that a zonotope with one more generator than its dimension is coloopless if and only if it is an LR zonotope. More strongly, we show in Section~\ref{sec:coloopless} that:
\begin{corollary}[Corollary~\ref{coro:contained-LRZ}]
\label{coro:contained-LRZ-intro}
Every coloopless zonotope $Z$ contains a lonely runner zonotope $Z'$ of the same dimension and with the same center. 
\end{corollary}

Corollary~\ref{coro:contained-LRZ-intro} has two important consequences.

\subsubsection*{1. It puts the Lonely Runner Conjecture in the broader context of coloopless zonotopes, a class that is more natural than that of LR zonotopes.}
Indeed:

\begin{corollary}
\label{coro:corunners}
\label{coro:corunners-LR}
For each value of $n$, the following  are equivalent:
\begin{enumerate}[label=\upshape(\arabic*),ref=\thetheorem .\arabic*]
\item Conjecture~\ref{conj:LRC} (Lonely Runner Conjecture).
\item  Every LR zonotope $Z$ of dimension $n-1$ has $\cminimum(Z) \le \tfrac{n-1}{n+1}$.
\item  Every coloopless zonotope $Z$ of dimension $n-1$ has $\cminimum(Z) \le \tfrac{n-1}{n+1}$.
\item The maximum $\cminimum$ among all coloopless $(n-1)$-zonotopes is $\tfrac{n-1}{n+1}$.
\end{enumerate}
\end{corollary}

\begin{proof}
The equivalence of (1) and (2) is essentially Proposition~\ref{prop:gamma-kappa-mu}. The implication (3)$\Rightarrow$(2) follows from the fact that every LR zonotope is coloopless and the converse from Corollary~\ref{coro:contained-LRZ-intro} and the fact that $\cminimum$ monotonically decreases with respect to containment (Proposition~\ref{prop:contained-easy-LR}).
The equivalence of (3) and (4) comes from the well-known fact that $\gap(1,\dots,n)= \tfrac1{n+1}$, that is, $\cminimum(\zono((1,\dots,n)) = \tfrac{n-1}{n+1}$.
\end{proof}

This raises the question of whether there is an even larger (but natural) class of lattice $d$-zonotopes for which the maximum $\cminimum$ is still that of $\cminimum(\zono((1,\dots,d+1))$ (conjecturally,  $\tfrac{d}{d+2}$). 
To argue that the answer is no, in Section~\ref{sec:examples} we construct, for every prime dimension $d$:

\begin{enumerate}
\item A lattice $d$-zonotope $Z$ with $d+1$ generators, only one of which is a coloop, of width at least three (with respect to $\Z^d$) and with $\cminimum(Z) = \tfrac{d-1}{d}>\tfrac{d}{d+2}$ (Proposition~\ref{prop:almost-coloopless}).
\item Infinitely many  lattice $(d+1)$-zonotopes with only one coloop, 
of width at least three and with $\cminimum(Z) \ge \tfrac{d-1}{d}$ (Corollary~\ref{coro:many-wide2}).
\end{enumerate}

\subsubsection*{2. Corollary~\ref{coro:contained-LRZ-intro}  provides a more transparent and  general proof of Theorem~\ref{thm:finite-LR}.}
Let $\kappa_d$ denote the maximum $\cminimum$ among  LR zonotopes (equivalently, by Corollary~\ref{coro:contained-LRZ-intro}, among  coloopless zonotopes) of  dimension $d$. We prove:

\begin{theorem}[Theorem~\ref{thm:finite-LR-new-body}]
\label{thm:finite-LR-new}
For every $\ell \in \Z_{>0}$ and every coloopless $d$-zonotope $Z\subset \R^d$ with at least $\ell^d$ lattice points we have
$
\cminimum(Z) \le \kappa_{d-1} + \frac1{\ell}.
$
\end{theorem}

This is implicitly proved in \cite[Section 3]{Malikiosis2024LinExpCheckingLRC} for the case $\ell=\binom{d+2}2$ and restricted to LR zonotopes. Let us see how 
Theorem~\ref{thm:finite-LR} easily follows that from that case:

\begin{proof}[Proof of Theorem~\ref{thm:finite-LR} assuming Theorem~\ref{thm:finite-LR-new} for LR zonotopes and  $\ell=\binom{d+2}2$]
Let $n=d+1$, $\ell = \binom{d+2}2$
 and $Z$ be an LR-zonotope of dimension $d$ with volume vector $(v_1,\dots,v_n)$ with $\sum_{S\subseteq [n]} \vv_S > \ell^d =\binom{n+1}2^{n-1}$. Since $\sum_{S\subseteq [n]} \vv_S$ equals the number of lattice points in $Z$ (Eq.~\ref{eq:lrc-points}),
Theorem~\ref{thm:finite-LR-new} gives
\[
\cminimum(Z) \leq \kappa_{n-2} + \frac1{\ell} = 
\frac{n-2}{n} + \frac1{\binom{n+1}2} = \frac{n-1}{n+1},
\] 
where the first equality follows from the inductive hypothesis in Theorem~\ref{thm:finite-LR},  equivalent to $\kappa_{n-2}=\tfrac{n-2}{n}$.
\end{proof}

Our two selling points for  Theorem~\ref{thm:finite-LR-new} in its full generality are that:
\begin{itemize}
\item Having the statement for arbitrary $\ell$ makes it useful even if it turned out that $\kappa_{d-1} \ne \tfrac{d-1}{d+1}$ (that is, if Conjecture~\ref{conj:LRC} is false); we would still have a finiteness result on $\kappa_{d}$, except with a different bound.

\item The proof of Theorem~\ref{thm:finite-LR-new} is actually  simpler than the original one in \cite{Malikiosis2024LinExpCheckingLRC}. In that one the authors need to prove that every projection of an LR zonotope contains an LR zonotope (their Theorem 3.1, with a three-page proof), while we rely in the (easy) observation that every projection of a coloopless zonotope is coloopless itself (our Proposition~\ref{prop:cosimple-coloopless-projection}, with a one-paragraph proof).
\end{itemize}

\subsection*{Cosimple zonotopes and the shifted Lonely Runner Conjecture}

The shifted LRC has many analogies with the unshifted one. The goal of Section~\ref{sec:sLRC} is to explore and push these analogies. We think this is interesting even after we know that sLRC is false for two reasons: on the one hand our analysis may help find the `correct' conjecture (that is, the minimum $\sgap$ among velocity vectors of a given size). On the other hand the analysis may lead to new insight on the original LRC.

Our starting point here is~\cite[Theorem B]{Malikiosis2024LinExpCheckingLRC} of Malikiosis et al., which establishes the analogue of Theorem~\ref{thm:finite-LR} with the same exact volume bound for potential counterexamples. However, their statement needed to assume  that every two-dimensional rational vector configuration with no zero or opposite elements satisfies what they called the \emph{Lonely Vector Property}, or LVP for short (Definition~\ref{def:LVP}).
We have found that this assumption fails:

\begin{theorem}[Corollaries~\ref{coro:LVP-rectangle} and~\ref{coro:noLVP12}]
\label{thm:noLVPC}
Not all rational vector configurations have the LVP. For example, the property fails for a certain 2-dimensional vector configuration with 12 elements.
\end{theorem}

Even if the LVP is not always satisfied, in Section~\ref{sec:cosimple} we still perform the same analysis of the proof 
of~\cite[Theorem B]{Malikiosis2024LinExpCheckingLRC}  as we have done for \cite[Theorem A]{Malikiosis2024LinExpCheckingLRC}. The role of coloopless zonotopes is now played by \emph{cosimple zonotopes}, already introduced in~\cite{Malikiosis2024LinExpCheckingLRC}. They are characterized by any of the following properties, reminiscent to those in Proposition~\ref{prop:coloopless}.

\begin{proposition}[Propositions~\ref{prop:co},~\ref{prop:Gale}, and~\ref{prop:cowidth}]
\label{prop:cosimple}
Let  $U=\{\uu_1,\dots,\uu_n\}\subset \ZZ^d$ be integer vectors.
Then, the following conditions are equivalent:
\begin{enumerate}
\item There is a linear dependence $\sum_{i=1}^n \lambda_i \uu_i =0$ with $\lambda_i\ne 0$ for all $i$ and $|\lambda_i|\ne |\lambda_j|$ for all $i\ne j$ (i.e., $U$ is cosimple~(Definition~\ref{def:co}).
\item No linear hyperplane contains all but one of the $\uu_i$, and if a linear hyperplane $H$ contains all but two of them, say $\uu_i$ and $\uu_j$, then none of $\uu_i+ \uu_j$ or $\uu_i - \uu_j$  is in $H$.
\item The Gale transform $U^*$ of $U$ contains no zero vector and no two equal or opposite vectors.
\item $\zono[U]$ has width at least $3$ with respect to the lattice generated by $U$.
\end{enumerate}
\end{proposition}

Condition (1) implies that a zonotope with one more generator than its dimension is cosimple if and only if it is an sLR zonotope. This makes it natural to ask the following question, analogue to Corollary~\ref{coro:contained-LRZ-intro}:

\begin{question}
\label{question:cosimple}
Does every cosimple zonotope contain an sLR zonotope of the same dimension?
\end{question}

We do not know the answer, but we have partial results. We define a \emph{minor} of a zonotope $\zono[U]$ (Definition~\ref{def:deletion}) to be any zonotope $\zono[U']$ of the same dimension where $U'$ is obtained from $U$ by iterating any of two operations:  \emph{diagonal} operations which combine two generators into one (Definition~\ref{def:diagonal}), and  \emph{deletion} of  generators. (The proof of Corollary~\ref{coro:contained-LRZ-intro} uses only the diagonal operation). We show that:

\begin{enumerate}
\item A cosimple zonotope $Z$ with generators $U$ has a (proper) minor that is also cosimple if and only if the Gale dual $U^*$ has the Lonely Vector Property (Theorem~\ref{thm:minor-LVP}).
\item Every cosimple zonotope $Z$ properly contains a minor $Z'$  that has width at least three with respect to the lattice of $Z$ (Corollary~\ref{coro:contained-wide}).
\end{enumerate}

Since zonotopes of width at least three are cosimple (Proposition~\ref{prop:cosimple}(4)), Corollary~\ref{coro:contained-wide} is almost a positive answer Question~\ref{question:cosimple}. The only catch is that in Corollary~\ref{coro:contained-wide} the width of the minor $Z'$ is measured with respect to the lattice of the original $Z$. In fact, Theorems~\ref{thm:minor-LVP} and~\ref{thm:noLVPC} imply the existence of cosimple zonotopes without cosimple (proper) minors. In case Question~\ref{question:cosimple} has a positive answer, deletion and diagonal operations alone are not enough to prove it.


\section{Preliminaries}
\label{sec:prelim}

\subsection{Covering and central radii of convex bodies}
\label{sec:covering}

A \emph{convex body} in $\R^d$ is any compact convex set. Unless otherwise specified we assume that our convex bodies are  \emph{proper}; that is, they have non-empty interior.

\begin{definition}
\label{def:minima}
Let $K\subset \R^d$ be a proper convex body.
\begin{enumerate}[label=\upshape(\roman*),ref=\thetheorem .\roman*]
\item For each point  $\cc\in \R^d$ in the interior of $K$ we call  \emph{first $\cc$-minimum of $K$}, denoted $\cminimum_\cc(K)$, the minimum $\lambda\ge 0$ such that $\cc +\lambda(K-\cc)$ intersects $\Z^d$.

\item The \emph{covering radius of $K$}, denoted $\covering(K)$ is the minimum $\lambda>0$ such that $\lambda K + \Z^d$ covers $\R^d$.
\end{enumerate}
\end{definition}

That is, $\cminimum_\cc(K)$ is the smallest dilation factor so that the dilation (centered at $\cc$) of $K$ with this factor contains a lattice point; in particular, $\cminimum_\cc(K)=0$ if and only if $\cc\in \Z^d$.
In turn, $\covering(K)$ is the smallest dilation factor $\mu$ so that every translation of $\mu K$ contains a lattice point. The first $\cc$-minimum is particularly interesting if $K$ is \emph{$\cc$-symmetric}; that is, if $K-\cc=-(K-\cc)$. When this happens we denote it simply $\cminimum(K)$.

Both parameters have an interpretation via the \emph{Minkowski distance} (or \emph{gauge}) defined by the convex body $K$ and the point $\cc$. These are the following quasi-norm and corresponding quasi-distance in $\R^d$; for each $\xx,\yy\in \R^d$:
\[
||\xx||_{K,\cc} := \min\big\{\lambda\ge0: \xx \in \lambda\cdot (K-\cc)\big\},
\qquad
\dist_{K,\cc} (\xx,\yy):=||\yy-\xx||_{K,\cc}.
\] 

The quasi-distance is symmetric (that is, a distance) if and only if $K$ is $\cc$-symmetric.
We clearly have%
\begin{proposition} Let $K$ be a convex body and $\cc$ an interior point of $K$:
\begin{enumerate}[label=\upshape(\roman*),ref=\thetheorem .\roman*]
\label{prop:distance}
\item $\cminimum_\cc(K) = \dist_{K,\cc} (\cc,\Z^d),$
\item $\covering(K) = \max_{\pp\in \R^d} \dist_{K,\cc} (\pp,\Z^d).$
\qed
\end{enumerate}
\end{proposition}

This makes it obvious  that $\cminimum_\cc(K) \le \covering(K)$. 
\bigskip

\begin{remark}
The covering radius in part (ii) of Definition~\ref{def:minima} and Proposition~\ref{prop:distance} is a classical parameter in convex geometry (see, e.g., \cite{KannanLovasz}, or~\cite[p.~ 381]{gruber2007convex}). 
The first $\cc$-minimum is introduced here,  but it is closely related to the following:
the \emph{coefficient of asymmetry} of a point $\pp \in \int(K)$ is defined as the largest ratio $\tfrac{||\qq_1-\pp|| }{ ||\qq_2-\pp||}$ where $[\qq_1,\qq_2]$ is a segment containing $\pp$ and with $\qq_1,\qq_2 \in \partial K$. This is always at least $1$, with equality if and only if $K$ is $\pp$-symmetric.
Suppose $K$ is $\cc$-symmetric and $K\cap \ZZ^d \ne\emptyset$.
 If $\operatorname{ca}(K)$ denotes the smallest coefficient of asymmetry among all lattice points in $K$, then 
 \[
 \operatorname{ca}(K)= \tfrac{1+\cminimum_\cc(K)}{1-\cminimum_\cc(K)}
 \quad\text{ or, equivalently, }\quad
\cminimum_\cc(K) =\tfrac{\operatorname{ca}(K)-1}{\operatorname{ca}(K)+1}.
\]
These formulas are essentially Proposition 4 in \cite{BeckSchymura}, a paper devoted to bounding the coefficient of asymmetry of zonotopes.
\end{remark}

For future reference we record the obvious monotonicity of $\covering$ and $\cminimum_\cc$.

\begin{proposition}
\label{prop:contained-easy}
Let $K'\subset K$ be convex bodies contained in one another.
\begin{enumerate}[label=\upshape(\roman*),ref=\thetheorem .\roman*]
\item $\cminimum_\cc(K) \le \cminimum_\cc(K')$ for every $\cc$ in the interior of $K'$.
\label{prop:contained-easy-LR}
\item $\covering(K) \le \covering(K')$.
\label{prop:contained-easy-sLR}
\qed
\end{enumerate}
\end{proposition}

We also need to study their dependence under projection.
In the following statement and in the rest of the paper we call \emph{length} of a segment $[\pp,\qq]$ with rational direction its length with respect to the lattice. That is, 
\[
\length([\pp,\qq]) := \frac{||\qq-\pp||}{||\uu||}
\]
where $\uu\in \Z^d$ is a primitive vector parallel to $[\pp,\qq]$. 

\begin{proposition}
\label{prop:CSS-sum-bound}
Let $K \subseteq \RR^d$ be a convex body and let $\cc$ be an interior point in $K$.
Let $\pi : \RR^d \to \RR^{d-1}$ be a linear projection with $\pi(\ZZ^d)= \ZZ^{d-1}$. Finally, let $\ell$ be the length of the segment $\pi^{-1}(\pi(\cc)) \cap K$.
Then,
\begin{enumerate}[label=\upshape(\roman*),ref=\thetheorem .\roman*]
\item $\cminimum_\cc(K) \leq \cminimum_{\pi(\cc)}(\pi(K) ) + \tfrac1\ell $, and
\item $\covering(K) \leq \covering(\pi(K) ) + \tfrac1\ell .$
\end{enumerate}
\end{proposition}

\begin{proof}
To simplify notation, let $\cc'= \pi(\cc)$ and $K'=\pi(K)$. 

By definition of $\cminimum_{\cc'}$ there is a point $\pp'\in \Z^{d-1}$ such that
\[
\pp' \in \cc' + \cminimum_{\cc'}(K')\cdot( K' - \cc') = \pi\left( \cc + \cminimum_{\cc'}(K')\cdot( K - \cc)\right).
\]
Hence, there is a $\pp\in \pi^{-1}(\pp')$ (not necessarily in $\ZZ^d$) such that 
\[
\pp \in \cc + \cminimum_{\cc'}(K')\cdot( K - \cc).
\]

Now, by definition of $\ell$ we have that the segment $\tfrac1\ell \cdot ( K - \cc) \cap \pi^{-1}(\cc')$ has length one. Since this segment is parallel to and of the same length as the segment $\left(\pp + \tfrac1\ell \cdot ( K - \cc) \right)\cap \pi^{-1}(\pp')$ and since $\pi^{-1}(\pp')$ is a lattice line, we conclude that there is a lattice point $\pp_0$ in it.

Then, we have
\begin{align*}
\pp_0 \in \pp + \tfrac1\ell \cdot ( K - \cc) \subseteq &\ 
\cc + \cminimum_{\cc'}(K')\cdot( K - \cc) + \tfrac1\ell \cdot ( K - \cc) \\
= &\  
\cc + \left(\cminimum_{\cc'}(K') + \tfrac1\ell\right) \cdot ( K - \cc),
\end{align*}
which shows that $\cminimum_\cc(K) \le \cminimum_{\cc'}(K') + \tfrac1\ell$ and finishes the proof of part (i).

Part (ii) is proved with similar arguments. We omit details since it is also a special case of~\cite[Lemma~2.1]{codenottisantosschymura2019the} (see also~\cite[Prop. 2.7]{Malikiosis2024LinExpCheckingLRC}).
\end{proof}

In the case of interest to us $K$ is a lattice zonotope and $\cc$ its center. This implies that 
both $K$ and $\Z^d$ are $\cc$-symmetric, and $\cc\in \tfrac12\Z^d$. 
With the convention that in this case we omit $\cc$ and $\pi(\cc)$ from the notation, part one of the statement simplifies to
\begin{align}
\cminimum(K) \leq \cminimum(\pi(K) ) + \tfrac1\ell.
\label{eq:kappa-projection}
\end{align}

To make Proposition~\ref{prop:CSS-sum-bound} useful we need  $K$ to contain long segments.
One way to guarantee this property is via the number of lattice points:

\begin{proposition}[{\cite[Lemma 2.5]{Malikiosis2024LinExpCheckingLRC}}, see also {\cite[Theorem 2.1]{betkehenkwills1993successive}}]
\label{prop:long-projection}
Let $Z\subset \RR^d$ be a  zonotope with more than $\ell^{d}$ lattice points. Then, there is a linear projection $\pi: \RR^d \to \RR^{d-1}$ with $\pi(\ZZ^d) = \ZZ^{d-1}$ and such that the fiber of the center $\cc'=\pi(\cc)$ of $\pi(Z)$ has length $\ge \ell$.
\end{proposition}

\begin{proof}
The number of lattice points implies that $Z$ has two lattice points $\pp, \qq$ in the same class modulo $\ell \ZZ^{d}$, so that the segment $[\pp,\qq]$ has length at least $\ell$. Observe that the segment $s= [\cc + \tfrac12(\pp-\qq), \cc + \tfrac12(\qq-\pp)]$ is parallel to it, and of the same length.
This segment is still contained in $Z$ since its end-points are the mid-points of $[\pp, 2\cc-\qq]$ and $[\qq, 2\cc-\pp]$, respectively, where $2\cc-\pp, 2\cc-\qq\in Z$ by definition of $\cc$-symmetric. Hence $\length(s) = \length([\pp,\qq]) \ge \ell$.

Let $\bb_0 \in \Z^d$ be the primitive vector proportional to $\pp-\qq$ and complete it to a basis $\{\bb_0,\dots, \bb_{d-1}\}$ of $\Z^d$. The projection in the statement is the one with $\pi(\bb_0) = \oo$ 
and $\pi(\bb_i) = \ee_i$ for $i=1,\dots, d-1$, since it has $s\subset \pi^{-1}(\cc')\cap Z$.
\end{proof}

\begin{remark}[Relation to successive minima] 
Recall that the successive minima of a $\oo$-symmetric convex body $C$ are 
\[
\lambda_i(C):= \min\{\lambda >0: 
\dim\,\langle\lambda C \cap \Z^d\rangle_{\R}  \ge i\}.
\]
If $K$ is a $\cc$-symmetric convex body for a $\cc \in \tfrac12\Z^d$, then
\begin{enumerate}[label=\upshape(\roman*),ref=\thetheorem .\roman*]
\item $\lambda_1(K-K) \le \cminimum(K)$ unless $\cc\in \Z^d$ (in which case $\cminimum(K)=0$).
\item $\lambda_d(K-K) \le \covering(K) \le \sum_{i=1}^d \lambda_i(K-K)$. 
\qed
\end{enumerate}
Part (ii) can be found in \cite[Lemma 2.4]{KannanLovasz} and is sometimes called Jarn\'ik's inequality \cite{Jarnik1,Jarnik2}.
That the three inequalities may be strict or non-strict, and that $\cminimum(K)$ may be both smaller or bigger than $\lambda_d(K-K)$ is shown in   Table~\ref{table:kappas}, where  $Z_{v_1,v_2,v_3}$ denotes the LR-zonotope with velocity vector $(v_1,v_2,v_3)$.
 
\begin{table}
\[
\begin{array}{l|cc|cc}
\text{polytope } P& \lambda_1(P-P) & \lambda_2(P-P) & \cminimum(P) & \covering(P) \\
\hline
[0,1]^2 & {1} & {1} &{1} & {1} \\ 
(\tfrac12,\tfrac12)+\conv\{\pm e_1, \pm e_2\} & {1/2} & {1/2} &{1} & {1} \\ 
Z_{1,2,3} & {1/3} & {2/5} &{1/2} & {1/2} \\ 
Z_{1,2,4} & {1/3} & {1/3} &{1/3} & {3/7} \\ 
\end{array}
\]
\caption{Some $2$-zonotopes illustrating $\lambda_1$, $\lambda_2$, $\cminimum$ and $\covering$.}
\label{table:kappas}
\end{table}

%

\end{remark}


\subsection{Coloopless and cosimple zonotopes}
A \emph{vector configuration} is a finite multiset of vectors in $\R^d$.
We consider its vectors labeled.

The linear dependences of a vector configuration $U=\{\uu_i\}_{i\in S}$ labelled by a set $S$ form a linear subspace of $\R^S$ that we denote $\dep(U)$:
\begin{align*}
\dep(U):= &\big\{(\lambda_i)_{i\in S} : \sum_i \lambda_i \uu_i =\oo\big\} \subset \R^S.
\end{align*}

\begin{definition}Let $U\subset \Z^d$ be an integer vector configuration of rank $d$. We say that:
\label{def:co}
\begin{enumerate}[label=\upshape(\roman*),ref=\thetheorem .\roman*]
\item $U$ is \emph{coloopless}  if $\dep(U)$ contains a vector with no zero entry.
\label{def:coloopless}
\item  $U$ is \emph{cosimple} if $\dep(U)$ contains a vector with all entries nonzero and of different absolute values.
\label{def:cosimple}
\end{enumerate}
\end{definition}

The following is a more geometric characterization:

\begin{proposition} Let $U$ be a configuration spanning $\RR^d$.
\label{prop:co}
\begin{enumerate}[label=\upshape(\roman*),ref=\thetheorem .\roman*]
\item $U$ is coloopless if and only if there is no $\uu\in U$ such that $U\setminus \{\uu\}$ is contained in some linear hyperplane. (We call such a $\uu$ a \emph{coloop}).
\label{prop:coloopless-hyperplane}

\item $U$ is cosimple if and only if  it  has neither a coloop nor two elements  $\uu_1,\uu_2 \in U$ such that $U\setminus \{\uu_1,\uu_2\}$ and one of $\uu_1\pm \uu_2$ are contained in some linear hyperplane.
\label{prop:cosimple-hyperplane}
\end{enumerate}
\end{proposition}

\begin{proof}
Part (i) is easy and part (ii) is \cite[Lemma 5.3]{Malikiosis2024LinExpCheckingLRC}.
\end{proof}

There is a second linear subspace naturally associated to a vector configuration $U$, namely that of \emph{linear evaluations}:
\begin{align*}
\val(U):= &\big\{(f(\uu_i)_{i\in S} \in \R^S : f \text{ is a linear functional in } \R^n \big\}
\end{align*}
The subspaces $\dep(U)$ and $\val(U)$ of $\R^S$ are the kernel and the row-space of the matrix with columns $U$. They are orthogonal complements to one another and Gale duality exchanges them.

\begin{definition}
\label{defi:Gale}
Let $U=\{\uu_i\}_{i\in S}$ and $U^*=\{\uu^*_i\}_{i\in S}$ be two vector configurations with the same label set $S$ (and, hence, of the same cardinality). We say that they are \emph{Gale dual} to one another, or that one is a \emph{Gale transform} of the other, if
$\dep(U) =\val(U^*)$ or, equivalently, $\dep(U) = \dep(U^*)^\perp$.
\end{definition}

See, e.g., \cite[Chapter 4]{DLRS2010triangulations}) for more on Gale duality.
The Gale dual of $U$ is unique modulo rank-preserving linear transformation, and the dual of an integer vector configuration can be chosen to be integer too.

It follows easily from Proposition~\ref{prop:co} that:

\begin{proposition}
\label{prop:Gale}
Let $U$ and $U^*$ be Gale duals of one another. Then:
\begin{enumerate}[label=\upshape(\roman*),ref=\thetheorem .\roman*]
\item $U$ is coloopless if and only if $U^*$ does not contain the zero vector.
\label{prop:Gale-LR}
\item $U$ is cosimple if and only if $U^*$ does not contain the zero vector nor two vectors that are equal or opposite to one another. 
\label{prop:Gale-sLR}
\end{enumerate}
\end{proposition}

\begin{remark}
In matroid theory, a \emph{loop} is a zero vector, and a configuration (or its matroid) is called \emph{simple} if it has neither loops nor pairs of parallel elements (elements that are multiples of one another). Hence, we are using the word (co)loopless exactly in the matroid sense (``the Gale dual has no loops''), and the word (co)simple in a weaker sense where only multiples with factor $\pm 1$ are forbidden in the Gale dual.
\end{remark}

The same zonotope $Z$ can be generated by different vector configurations, but there is a unique choice of minimal cardinality: the one containing no parallel pairs.
We call such generating sets \emph{reduced}.
It is quite obvious via Proposition~\ref{prop:co} that if a given set of generators for a zonotope $Z$  is \emph{not} coloopless (respectively, cosimple), reducing it by combining parallel generators  cannot make it coloopless (respectively, cosimple). 
Thus, it makes sense to define coloopless and cosimple zonotopes as follows.

\begin{definition}
\label{def:cozono}
A lattice zonotope is  \emph{coloopless} (respectively, \emph{cosimple}) if its reduced set of generators is coloopless (respectively, cosimple). Equivalently, if it admits a coloopless (respectively, cosimple) set of generators.
\end{definition}


%
\subsection{Relation to width}

To further show that the definitions of cosimple and coloopless are natural, we relate them with the notion of width.

Recall that the \emph{width} of a convex body $K\subset \R^d$ with respect to a linear functional $f \in(\R^d)^*$ is the length of the interval $f(K)$; put differently:
\[
\width(K,f):=\max_{\pp\in K}(f(\pp)) -\min_{\pp\in K}(f(\pp)).
\]
If $Z$ is a zonotope with generators $U$, then
\begin{align}
\width(Z,f) = \sum_{\uu\in U} |f(\uu)|.
\label{eq:width}
\end{align}

Let $\Lambda\subset \R^d$ be a lattice; typically, but not necessarily, $\Lambda=\Z^d$.
We call \emph{lattice width} (or just \emph{width}) of a convex body $K$ with respect to $\Lambda$ the minimum width it has with respect to non-zero \emph{lattice functionals}. That is,
\[
\width_\Lambda(C) := \min_{f \in \Lambda^*\setminus \{\oo\}} \width(C,f),
\]
where $\Lambda^* = \{f\in (\R^d)^* : f(\Lambda) \subset \Z\}$.

Part (ii) of the following statement is essentially \cite[Corollary 5.6]{Malikiosis2024LinExpCheckingLRC}.

\begin{proposition}
\label{prop:cowidth}
Let $Z$ be a lattice zonotope,  $U$ its reduced generating set,  and  $\Lambda:=\Z[U] \le \Z^d$  the lattice generated by $U$.

\begin{enumerate}[label=\upshape(\roman*),ref=\thetheorem .\roman*]
\item
\label{prop:cowidth-LR}
 $\text{ $Z$ is coloopless }   \Leftrightarrow\     \width_\Lambda(Z) \ge 2  \    \Rightarrow\     \width_{\Z^d}(Z) \ge 2$.
\item
\label{prop:cowidth-sLR}
$ \text{ $Z$ is cosimple } \ \,    \Leftrightarrow\    \width_\Lambda(Z) \ge 3   \      \Rightarrow\     \width_{\Z^d}(Z) \ge 3$.
\end{enumerate}
\end{proposition}

\begin{proof}
In both parts the  implication ``$\width_\Lambda(Z) \ge k    \Rightarrow   \width_{\Z^d}(Z) \ge k$'' is obvious since $\Lambda\le \Z^d$ and finer lattices produce greater width. 
Hence we concentrate on the  double implication.

For the double implication we use the characterization of coloopless/cosimple in Proposition~\ref{prop:co} and Equation~\eqref{eq:width} for the width of a zonotope in terms of its generators: 
\begin{enumerate}
\item If $Z$ is not coloopless, then there is $\uu \in U$ a coloop. That is, $U\setminus \{\uu\}$ is contained in a hyperplane and then $Z$ has width one (in the lattice $\Lambda$) with respect to a functional vanishing on $U\setminus \{\uu\}$.
Conversely, by Equation~\eqref{eq:width} if $\width_{\Lambda}(Z,f)=1$ for some $f\in \Lambda^*\setminus \{\oo\}$, there is a $\uu\in U$ with $f(\vv)=0$ for every $\vv\in U\setminus \{\uu\}$ (and necessarilly $f(\uu)=\pm 1$). Hence $\uu$ is a coloop.

\item If $Z$ is not cosimple, then there are $\uu_1, \uu_2\in U$ such that $U\setminus \{\uu_1,\uu_2\}$ is contained in a hyperplane (let $f$ be a functional vanishing in it), and $\uu_1$ and $\uu_2$ are such that $|f(\uu_1)|=|f(\uu_2)|$ (they are at the same lattice distance from the hyperplane, on the same side or different sides of it). By scaling of $f$, and by Equation~\eqref{eq:width}, $Z$ has width $2$ with respect to $\Lambda$.
Conversely, and again by Equation~\eqref{eq:width}, if $\width_{\Lambda}(Z,f)=2$ for some $f\in \Lambda^*\setminus \{\oo\}$, the only possibility is for $\uu_1, \uu_2\in U$ to exist so that $f$ vanishes in $U\setminus \{\uu_1,\uu_2\}$ and $|f(\uu_1)|=|f(\uu_2)|=1$, hence the result.

\end{enumerate}

\end{proof}

One important property that we need is that the classes of coloopless and cosimple zonotopes  are closed under projection:

\begin{proposition}
\label{prop:cosimple-coloopless-projection}
Let $Z$ be a coloopless (resp. cosimple) $d$-zonotope, and let $\pi:\R^d \to \R^k$ be linear. Then $\pi(Z)$ is also coloopless (resp. cosimple).
\end{proposition}

\begin{proof}
Let $U$ be the reduced set of generators of $Z$ and let $Z'=\pi(Z)$,  $U':=\pi(U)$. If $U'$ was not coloopless, there would be a hyperplane in $\R^k$ containing all but one of its generators. This would lift to a hyperplane in $\R^d$ containing all but one of the generators of $U$.
Analogously, if $U'$ was not  cosimple, a hyperplane in $\R^k$ containing all but two of the generators of $U'$ and containing the sum or difference of the other two would lift to a hyperplane with the same property for $U$.
\end{proof}


\section{LRC and coloopless zonotopes. Revisiting the finiteness bound}
\label{sec:LRC}

\subsection{Every coloopless zonotope contains an LR-zonotope.}
\label{sec:coloopless}

We here show that LR zonotopes are the minimal (under containment) coloopless polytopes. In particular, by the monotonicity of $\cminimum$ under containment (Proposition~\ref{prop:contained-easy}), the maximum $\cminimum(Z)$ among all coloopless zonotopes of a given dimension is always attained at some LR zonotope.

To prove this we introduce the following operation that combines two generators of a zonotope into one, hence decreasing the number of them.

\begin{definition}[Diagonal subzonotopes]
\label{def:diagonal}
Let $U\subset \Z^d$ be a vector configuration.
For each $\uu, \vv\in U$, the \emph{positive diagonal} of $U$ at $\uu,\vv$ is the configuration $U \setminus \{\uu,\vv\} \cup \{\uu+\vv\}$. Similarly, the \emph{negative diagonal} is $U \setminus \{\uu,\vv\} \cup \{\uu-\vv\}$. We denote them $U_{\uu,\vv}^+$ and $U_{\uu,\vv}^-$, respectively.

A \emph{diagonal subzonotope} of a zonotope $Z$ is any zonotope of the same dimension as $Z$ and obtained by iterating the diagonal operation on the reduced generators of $Z$.
\end{definition}

\begin{lemma}
\label{lemma:contained-diagonal}
Let $Z'$ be a diagonal subzonotope of $Z$. There is a subset $S$ of generators of $Z$ such that $Z' + \sum_{\uu\in S} \uu$ is contained in $Z$ and has the same center.
\end{lemma}

\begin{proof}
By induction on the number of diagonal steps needed to go from $Z$ to $Z'$, we only need to show the statement for the case of a single step.

When the diagonal is chosen positive, that is, $Z'=\zono[U_{\uu\vv}^+]$ for some $\uu, \vv\in U$, the result is obvious with $S=\emptyset$, since the segment generated by $\uu+\vv$ is contained in $Z$ and contributes to the center $\frac{1}{2}(\uu + \vv)=\frac{1}{2}\uu + \frac{1}{2}\vv$.

In the case of the negative diagonal, 
let $_{-v}U:=U \setminus \{\vv\}  \cup \{-\vv\}$  and consider corresponding zonotope $Z[_{-v}U]$. We clearly have that $Z[_{-v}U]= Z -\vv$. If we apply to $Z[_{-v}U]$ the same diagonal process that produced $Z'$ from $Z$, we get the same zonotope $Z'$ but now using a positive diagonal of $Z[_{-v}U]$. Hence, the positive case gives that $Z'$ is contained in and has the same center as $Z -\vv$, as we wanted to show.
\end{proof}

\begin{figure}[htb]
\centerline{\includegraphics[scale=.5]{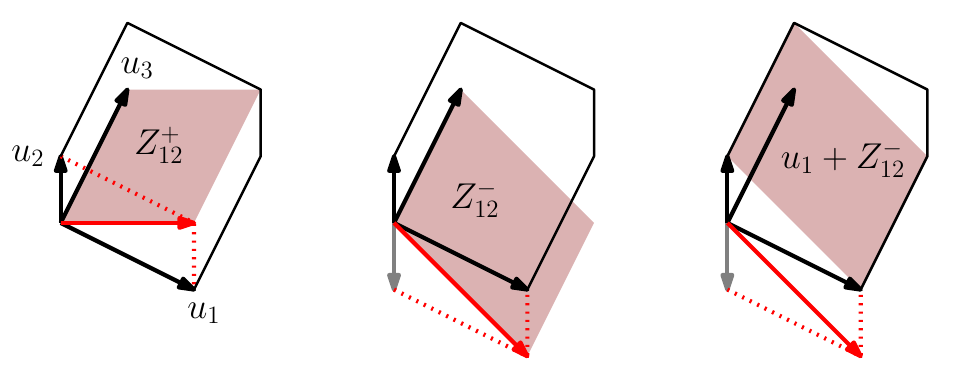}}
\caption{Illustration of the proof of Lemma~\ref{lemma:contained-diagonal}}
\end{figure}

\begin{lemma}
\label{lemma:Gale-diagonal}
Let $U$ be a coloopless configuration and  $U^*$ a Gale dual of it. Let $\uu, \vv\in U$, and let $\uu^*, \vv^*$ be the corresponding elements in $U^*$. Then:
$U_{\uu,\vv}^\pm$ is coloopless if, and only if, no element of $U^*$ is parallel to $\uu^*\mp\vv^*$.
\end{lemma} 

\begin{proof}
To fix notation, write $U=\{\uu_1, \dots, \uu_{n-2}, \uu, \vv\}\subset \Z^d$ and let $U^*=\{\uu^*_1, \dots, \uu^*_{n-2}, \uu^*, \vv^*\}\subset \Z^{n-d}$ be its Gale dual.
For the forward direction suppose that $U_{\uu,\vv}^+$ is not coloopless (the case of $U_{\uu,\vv}^-$ is analogous).  There are two possibilities:
\begin{itemize}
\item \emph{$\uu + \vv$ is a coloop in $U_{\uu,\vv}^+$}. Then, there is a non-zero linear functional $f$ with $f(\uu_i)=0$ for all $i$. In the Gale dual, $f$ translates to a linear dependence involving only $\uu^*$ and $\vv^*$. Hence, $\uu^*$ and $\vv^*$ are parallel to one another, which makes them parallel to  $\uu^*-\vv^*$.

\item \emph{A certain $\uu_i$ is a coloop in $U_{\uu,\vv}^+$}. Then, there is a non-zero linear functional $f$ with $f(\uu_i)\ne 0$, $f(\uu_j)=0$ for every $j\ne i$, and  $f(\uu+\vv)= 0$; that is,  $f(\uu)=-f(\vv)$. In the Gale dual, $f$ translates to a linear dependence involving only $\uu_i^*$, $\uu^*$ and $\vv^*$, and with opposite coefficients in $\uu^*$ and $\vv^*$. Hence, $\uu^* - \vv^*$ is proportional to $\uu_i^*$.
\end{itemize}

For the converse, observe that all the arguments above can be reversed: parallelism in $U^*$ is equivalent to existence of a certain linear dependence, which translates to a functional in $U$, which translates to colooplessness.
\end{proof}

\begin{corollary}
\label{coro:contained-LR}
If $U$ is a coloopless configuration with at least two more elements than its rank then there are  $\uu,\vv\in U$ such that one of $U_{\uu,\vv}^+$ or $U_{\uu,\vv}^-$ is coloopless.
\end{corollary}

\begin{proof}
By hypothesis, the Gale dual $U^*$ has rank at least two, and it has no loop. We want to find two elements in $U^*$ satisfying one of the conditions of Lemma~\ref{lemma:Gale-diagonal}. This is the same as two elements $\uu^*, \vv^*\in U^*\cup (-U^*)$  such that no element of $U^*$ is parallel  to $\uu^* + \vv^*$. One such choice is to take  two vectors of $U^*\cup (-U^*)$ forming the smallest (but non-zero) angle.
\end{proof}

\begin{corollary}
\label{coro:contained-LRZ}
Every coloopless zonotope contains an LR zonotope of the same dimension and with the same center.
\end{corollary}

\begin{proof}
Let $Z$ be a coloopless zonotope. Iterating the previous corollary we find a coloopless diagonal $Z'$ of $Z$ with one more generator than its dimension, that is, a diagonal that is an LR zonotope.
By Lemma~\ref{lemma:contained-diagonal}, there is an integer translation of $Z'$ that is contained in $Z$ and with the same center.
\end{proof}

We are now ready to prove Theorem~\ref{thm:finite-LR-new}, which we state again:

\begin{theorem}[Theorem~\ref{thm:finite-LR-new}]
\label{thm:finite-LR-new-body}
Let $Z\subset \R^d$ be a coloopless $d$-zonotope with at least $\ell^d$ lattice points
for a certain $\ell \in \Z_{>0}$. Then,
$
\cminimum(Z) \le \kappa_{d-1} + \frac1{\ell}.
$
\end{theorem}

\begin{proof}
Proposition~\ref{prop:long-projection} gives us a linear projection in the hypotheses of 
Proposition~\ref{prop:CSS-sum-bound}(i) (in the symmetric version of \eqref{eq:kappa-projection}) which in turn gives that
\[
\cminimum(Z) \leq \cminimum(\pi(Z) ) + \frac1{\ell}.
\]
Now, the zonotope $\pi(Z)$ is coloopless, since it is a linear projection of a coloopless zonotope (see Proposition~\ref{prop:cosimple-coloopless-projection}),
and by Corollary~\ref{coro:contained-LR} there is an LR $(d-1)$-zonotope $Z'$ contained in $\pi(Z)$ and with the same center. By Proposition~\ref{prop:contained-easy} and inductive hypothesis  we have
\[
\cminimum(Z) \leq \cminimum(\pi(Z) ) + \frac1{\ell} \le 
\cminimum(Z' ) + \frac1{\ell}  \le 
\kappa_{d-1} + \frac1{\ell}.
\qedhere
\] 
\end{proof}

\subsection{Some non-coloopless (counter)-examples}
\label{sec:examples}
It follows from our results in Section~\ref{sec:coloopless} that, although the LRC apparently only deals with LR zonotopes, it is equivalent to the same statement for all coloopless zonotopes (Corollary~\ref{coro:corunners-LR}). One may ask whether this can be extended further. 

For example, Lemma~\ref{prop:cowidth} says that the class of coloopless zonotopes is very close to that of lattice zonotopes of width larger than one; that is, lattice zonotopes that are not the lattice Cartesian product of a lattice segment with a lower-dimensional zonotope. 

In this section we show examples indicating that coloopless zonotopes may indeed be the widest natural class to be considered.

We first look at parallelepipeds, which are never coloopless.
The problem of maximizing $\cminimum$ among all lattice parallelepipeds of a given dimension is called the ``Lonely Rabbit Problem'' (see \cite{BeckSchymura}). The following example  was essentially known to Wills  in 1968 (see~\cite[Lemma 11]{willslrc2}), although expressed as a question in diophantine approximation.

\begin{proposition}
\label{prop:parallelepiped}
Suppose that $2d+1$ is a prime number. Let $Z$ be the $d$-dimensional parallelepiped represented as the unit $d$-cube with respect to the lattice $\Lambda:= \Z^d + \frac1{2d+1} \langle(1,\dots,d)\rangle$.

Then, $\width(Z)=3$ and $\cminimum(Z)= \frac{2d-1}{2d+1}$. In particular, $\cminimum(Z) > \frac{d}{d+2}$.
\end{proposition}

\begin{proof}
Let $\pp_1=\frac1{2d+1} (1,\dots,d)$ and, for each $i\in \Z$, $\pp_i = \fractional{i\, \pp_1}$, by which we mean the fractional part of $i\, \pp_1$: each coordinate of $i \pp_1$ is reduced modulo $\Z$ to lie in $[0,1)$. Observe that the index $i$ is important only modulo $2d+1$, and that $\pp_0=\oo$. Primality of $2d+1$ makes the other $2d$ points $\pp_1,\dots,\pp_{2d}$ to have all their coordinates in $(0,1)$, hence to lie in the interior of $Z$. In fact, these $2d$ points are the only interior lattice points in $Z$.

Moreover, for every $i$, the $d$ coordinates of $\pp_{i}$ cover the $d$ non-zero pairs of opposite 
values $\tfrac{j}{2d+1}$ and $\tfrac{2d+1-j}{2d+1}$, $j\in \{1,\dots,d\}$, in a cyclic manner. In particular, the $2d$ $\pp_i$ for $i\in \{1,\ldots, 2d\}$ interior points are all at the same distance from the center of $Z$, namely at distance 
\[
\frac{\tfrac12 - \tfrac1{2d+1}}{\tfrac12} = 
{1 - \frac2{2d+1}} = 
\frac{2d-1}{2d+1},
\]
where distance is defined via the gauge of $Z$, which in our setting coincides with the $L_{\infty}$ distance. This shows $\cminimum(Z)= \frac{2d-1}{2d+1}$.

For the width:
\begin{itemize}
\item The functional $2x_1-x_2$ takes only integer values in $\Lambda$, since it is a lattice functional and it vanishes at $\pp_1$. It gives width three to the unit cube, which shows $\width_\Lambda(Z) \le 3$.

\item No nonzero functional in $\Lambda^*$ gives width two or one to the unit cube, because $\Lambda^* \subset (\Z^d)^*$ and the only integer functionals giving width one or two to the unit cube are $e_i$ and $e_i \pm e_j$; none of them takes an integer value at $\pp_1$.
\qedhere
\end{itemize}
\end{proof}

\begin{example}
As the first cases of this construction we have:
\begin{itemize}
\item  The lattice parallelogram represented as the unit square with respect to the lattice $\Z^2 + \frac15 \langle(1,2)\rangle$. It has area $5$ and $\cminimum=\frac35> \frac 12$. See \cite[Figure 3]{Malikiosis2024LinExpCheckingLRC}.

It is worth mentioning that this parallelogram is in fact the \emph{only} $2$-zonotope of width greater than one and $\cminimum>\tfrac12$. (The lattice $2$-zonotopes with $\covering>\tfrac12$, which is weaker than $\cminimum>\tfrac12$, have been completely classified in \cite[Theorem 6.3]{Malikiosis2024LinExpCheckingLRC}). 
\item  The lattice $3$-parallel\-e\-piped represented as the unit cube with respect to the lattice $\Z^3 + \frac17 \langle(1,2,3)\rangle$. It has volume $7$ and $\cminimum=\frac57 > \frac35$.
\end{itemize}
\end{example}

\begin{corollary}
\label{coro:many-wide}
Let $2d-1$ be prime. 
Then, there are infinitely many $d$-dimensional lattice parallelepipeds of width at least three and with $\cminimum \ge \tfrac{2d-3}{2d-1}$.
In particular, $\cminimum > \frac{d}{d+2}$ for $d\ge 4$. 
\end{corollary}

\begin{proof}
Consider the example $Z$ of Proposition~\ref{prop:parallelepiped} in dimension $d-1$. (This works since $2(d-1)+1 = 2d-1$ is a prime).
Let $Z_k = Z \times [0,k] \subset \R^d$. Then, for sufficiently large $k$ we have that $\width(Z_k)= \width(Z) \ge 3$ and
\[
\cminimum(Z_k) \ge \cminimum(Z)= \frac{2(d-1)-1}{2(d-1)+1} = \frac{2d-3}{2d-1} .
\qedhere
\]
\end{proof}

\begin{remark}
\label{rem:cusick}
The example  of Proposition~\ref{prop:parallelepiped} was generalized by Cusick~\cite{cusicksimultaneous} as follows. Instead of assuming $2d+1$ to be prime, consider any positive integer $q$.

If $q$ is prime, let $d=\tfrac{q-1}2 = \tfrac{\phi(q)}2$ and use the construction above.

If $q$ is composite, let 
\[
d=\tfrac{\phi(q)}2 + h(q),
\]
where $\phi$ denotes Euler's totient function (which is always even) and $h(q)$ is the number of prime factors of $q$. Consider $\R^{d} = \R^{\tfrac{\phi(q)}2} \times \R^{h(q)}$ and modify the construction  taking 
\[
\pp_1 =(\tfrac{a_1} q, \dots, \tfrac{a_{{\phi(q)}/2 }} q, \tfrac1{p_1}, \dots,  \tfrac1{p_{h(q)}} ),
\]
where the $a_i\in [1,q]$ are representatives for the $\tfrac{\phi(q)}2$ pairs of opposite primitive classes modulo $q$, and  ${p_1}, \dots,  p_{h(q)}$ are the primes dividing $q$. 
This makes every $\lfloor i \pp_i \rfloor$ to have either a coordinate $\tfrac{ia_j} q \in \big\{\tfrac1q, 1-\tfrac1q\big\}$ 
or a coordinate $\tfrac{i}{p_j}=0$; the former happens if $\gcd(i,q)=1$ and the latter if not.
The proof goes through to show that the parallelepiped $Z$ obtained has
\[
\cminimum(Z) = 1 - \frac2q.
\]

Asymptotically, Cusick~\cite{cusicksimultaneous} conjectured and Schark~\cite{Schark} proved that  the minimum $1-\cminimum$ among all parallelepipeds of a given dimension $d$ is the one obtained by this construction with $q$ being the product of the smallest primes.
This minimum $1-\cminimum$ grows as  $\Theta(\tfrac1{d \log\log d})$ rather than the $\frac{1}{d+2}$ conjectured for LR zonotopes.
\end{remark}

Our second example is original, although inspired by the previous one. It is \emph{almost coloopless} in the sense that only one of its generators is a loop. 

\begin{proposition}
\label{prop:almost-coloopless}
Assume that $d$ is a prime greater than $2$ and consider the  $d$-zonotope $Z$ with the $d+1$
generators
\[
\uu_i= e_i, \ i\in\{1,\dots,d-1\},
\quad
\uu_d= -\sum_{i=1}^{d-1} e_i,
\quad
\vv = e_d,
\]
of which only $\vv$ is a coloop.

When considered with respect to the lattice
\[
\Lambda := \Z^d +   \frac{1}{d}\langle(1,2,\ldots,d-1,1)\rangle
\]
 $Z$ has width at least three and $\cminimum(Z) = \tfrac{d-1}{d} >\frac{d}{d+2}$.
\end{proposition}

\begin{proof}
With respect to the integer lattice $\Z^d$, $Z$ is the cartesian product of a unit segment and the LR $(d-1)$-zonotope with volume vector $(1,\dots,1)$. 
The last generator $\vv$ corresponds to the segment factor and we give a different notation to it since it plays a completely different role in the construction. 
For the same reason, when writing coordinates of points in this example we do so in the form  $(x_1,\dots, x_{d-1}; y)$, separating the last coordinate from the rest and using for it the letter $y$ instead of $x_{d}$.
With this convention, $Z$  can be defined by the following $2\big(\binom{d}{2} +1\big)$ inequalities:
\begin{align*}
|x_i| &\le 1, \ i\in\{1,\dots,d-1\}, \\
|x_i-x_j| &\le 1, \ i,j\in\{1,\dots,d-1\}, \\
0\le y &\le 1.
\end{align*}

The symmetry group of $Z$ contains the whole group of permutations of the first $d$ generators $\uu_i$, but we are interested in a smaller group, cyclic of order $d$ and generated by the following linear map which cyclically permutes $\uu_1 , \uu_2 , \dots , \uu_d $:
\begin{align}
(x_1,\dots,  x_{d-1}; y) \longmapsto (-x_{d-1}, x_1-x_{d-1},\dots,x_{d-2}-x_{d-1}; y).
\label{eq:symmetry}
\end{align}
Observe that  $Z$ decomposes into  $d$ parallelepipeds: the unit cube $[0,1]^d$ (generated by $\uu_1,\dots,\uu_{d-1}, \vv$) and its $d-1$ images under the action of this symmetry.

With respect to the standard lattice $\Z^d$ our zonotope has width one, hence it has $\cminimum=1$. But we consider it with respect to the finer lattice
\[
\Lambda := \Z^d +  \langle \pp_1\rangle,
\]
with $\pp_1:=  \tfrac1d (1,2,\dots,d-1; 1)$. 
Observe that $\Lambda$ is still invariant under the cyclic symmetry~\eqref{eq:symmetry}, since the image of $\pp_1$ under this symmetry is 
\[
\tfrac1d(1-d, 2-d,\dots, -1; 1) = \pp_1 - (1,\dots,1;0).
\]

$\Lambda$ is a superlattice of $\Z^d$ of index $d$, and the $d-1$ lattice points inside the unit cube are  (with the notation of the previous example) the points
\[
\pp_i : = \fractional{i\, \pp_1}, \ i=1,\dots,d-1. 
\]
The fact that $d$ is a prime implies that the coordinates of each $\pp_i$ are a permutation of those of $\pp_1$. In particular, every $\pp_i$ has one coordinate $x_{\sigma(i)}$ equal to $\tfrac{d-1}{d}$ which, taking into account the facet inequalities $-1\le x_{\sigma(i)}\le 1$ in the definition of $Z$, implies that this point does not meet the interior of $\cc + \tfrac{d-1}{d}(Z-\cc)$, where $\cc=(\oo; \tfrac12)$ is the center of $Z$. By symmetry, the same happens for the lattice points in the other $d-1$ parallelepipeds making up $Z$. This implies that
\[
\cminimum(Z) \ge \frac{d-1}{d},
\]
as claimed. The fact that one of the coordinates equals $\frac{d-1}{d}$ (e.g., in $\pp_1$) implies that this is an equality. 

The proof that $\width_{\Lambda}(Z)=3$ is similar to the one in Proposition~\ref{prop:parallelepiped}:
\begin{itemize}
\item The functional $x_1-y$ takes only integer values on $\Lambda$, since it is a lattice functional and it vanishes at $\pp_1$. By \eqref{eq:width} its total width on $Z$ is three, since it takes value $\pm 1$ at $\uu_1$, $\uu_d$ and $\vv$, and it vanishes on the rest of generators of $Z$.

\item A functional  giving width one or two to $Z$ would do the same to the unit cube. 
Since  $\Lambda^* \subset (\Z^d)^*$ (because the generators of $Z$ span $\ZZ^d$), such functional needs to be  either a single coordinate or a sum or difference of two coordinates. Among these, the only one with integer value on $\pp_1$ is $x_1-y$ which, as we have seen, gives width three to $Z$.
\qedhere
\end{itemize}
\end{proof}

\begin{remark}
The example of Proposition~\ref{prop:almost-coloopless} can be generalized in a way similar to Remark~\ref{rem:cusick}. The main difference is that we now use the construction for $q$ composite also when $q$ is a prime, with $h(q)=1$ in this case.
That is, for an arbitrary $q$ let
$d=\phi(q) + h(q)$,
divide $\R^{d} = \R^{\phi(q)} \times \R^{h(q)}$, and modify the construction above by using 
\[
\pp_1 =(\tfrac{a_1} q, \dots, \tfrac{a_{\phi(q) }} q, \tfrac1{p_1}, \dots,  \tfrac1{p_{h(q)}} ),
\]
where the $a_i\in [1,q]$ are representatives for the ${\phi(q)}$ primitive classes modulo $q$ and  ${p_1}, \dots,  p_{h(q)}$ are the primes dividing $q$. 
This makes every $\lfloor i \pp_i \rfloor$ to have either a coordinate $x_j = \tfrac{q-1}q$ or a coordinate $y_j=0$, and the proof goes through to show that the zonotope $Z$ obtained has
\[
\cminimum(Z) = 1 - \frac1q.
\]
\end{remark}

With the same proof as in Corollary~\ref{coro:many-wide}, this example implies:

\begin{corollary}
\label{coro:many-wide2}
Let $p=d-1$ be prime. Then, there are infinitely many $d$-dimensional lattice zonotopes with only one coloop, 
of width at least three and with $\cminimum \ge \tfrac{d-2}{d-1}$.
\end{corollary}


\section{sLRC, cosimple zonotopes, and the Lonely Vector Property}
\label{sec:sLRC}

\subsection{Minimal cosimple polytopes and the Lonely Vector Property}
\label{sec:cosimple}

We have proved in Section \ref{sec:coloopless} that every \emph{minimal} coloopless polytope is an LR zonotope. The natural notion of containment related to the LRC is center-preserving containment, and  we saw that a restricted version where only ``diagonal containment'' is considered suffices  (Corollary~\ref{coro:contained-LRZ}).

In the light of the parallelism between coloopless zonotopes in relation to the LR and cosimple ones in relation to the sLR, one could expect every minimal cosimple zonotope to be an sLR zonotope. We do not know whether that holds, but we can prove that if containment is restricted to the natural extension of the concept of diagonals that we used for coloopless zonotopes then the answer is no. 

The extension is that, since in the context of the sLRC it is not a problem to change the center of the zonotope, besides taking diagonals we consider the operation of deleting generators.

\begin{definition}[Deletion and minors]
\label{def:deletion}
Let $U\subset \Z^d$ be a vector configuration, generating a zonotope $Z=\zono[U]$.
For each $\uu\in U$, the \emph{deletion} of $\uu$ in $U$ is the configuration $U \setminus \{\uu\}$. We denote it $U_\uu$.

We call \emph{minor} of a zonotope $Z$ any zonotope of the same dimension as $Z$ that is obtained by iterating the diagonal or deletion operations on the reduced generators of $Z$.
\end{definition}

Observe that our definition of minor is not the same as the one in matroid theory. Our deletion is the same as the matroid theoretic one, but our diagonals do not have a clear matroidal counterpart. 

\begin{lemma}
\label{lemma:contained-minor}
If $Z'$ is a minor of $Z$ then there is a subset $S$ of generators of $Z$ such that $Z' + \sum_{\uu\in S} \uu$ is contained in $Z$.
\end{lemma}

\begin{proof}
For the diagonal operation we proved this in Lemma~\ref{lemma:contained-diagonal}. For the deletion it is obvious (and no translation is needed).
\end{proof}

In what follows, if $S$ is a vector configuration we denote by $\double{S}$ the following configuration of the same rank:
\begin{itemize}
\item For each $\uu \in S$, we include $2\uu$ in $\double{S}$.
\item For each $\{\uu, \vv\} \in \binom{S}{2}$, we include $\uu+ \vv$ and one of $\{\uu-\vv, \vv-\uu\}$ in $\double{S}$.
\end{itemize}
$\double{S}$ is considered as a multiset and it has $|S|^2$ elements.

\begin{proposition}
\label{prop:minor-LVP}
Let $U$ be a vector configuration, with Gale dual $U^*$. Then:
\begin{enumerate}
\item $U_\uu$ is cosimple if and only if $\double{(U^*)}$ contains no element proportional to $2\uu^*$.
\item $U^+_{\uu,\vv}$ is cosimple if and only if $\double{(U^*)}$ contains no element proportional to $\uu^*-\vv^*$.
\item $U^-_{\uu,\vv}$ is cosimple if and only if $\double{(U^*)}$ contains no element proportional to $\uu^*+\vv^*$.
\end{enumerate}
\end{proposition}

\begin{proof}
Recall that a configuration $U$ is \emph{not} cosimple if and only if there is a non-zero functional $f$ that vanishes in all but one of its elements, or that vanishes in all but two and has the same absolute value on those two (Proposition~\ref{prop:co}). Recall also that the vectors $(f(\uu))_{\uu \in U}$ are exactly the linear dependencies in the Gale transform $U^*$ (Definition~\ref{defi:Gale}).

For part (1), suppose that  $\double{(U^*)}$ contains an element proportional to $\uu^*$. There are two possibilities:
\begin{itemize}
\item The element is of the form $2 \vv^*$ for a $\vv\in U\setminus \{\uu\}$. Then, $\uu^*$ and $\vv^*$ are linearly dependent in $U^*$, which translates to a linear function $f$ that vanishes on $U\setminus \{\uu,\vv\}$. That linear function proves that $\vv$ is a coloop in $U_\uu=U\setminus \{\uu\}$. Conversely, if a $\vv\in U_\uu =U\setminus \{\uu\}$ is a coloop then there is a functional $f$ vanishing on 
$U\setminus \{\uu,\vv\}$ and with $f(\vv)\ne 0$. In $U^*$ this produces a linear dependence involving only $\uu$ and $\vv$, so indeed $2\uu$ and $2\vv$ are parallel elements in $\double{(U^*)}$.

\item The element is of the form $\vv^*\pm \ww^*$ for  $\vv,\ww\in U\setminus \{\uu\}$. Then, $\uu^*$ and $\vv^*\pm \ww^*$ are linearly dependent in $U^*$, which translates to a linear function $f$ that vanishes on $U\setminus \{\uu,\vv,\ww\}$ and with $f(\vv) = \pm f(\ww)$. Then, $\vv$ and $\ww$ prove that $U_\uu = U\setminus \{\uu\}$ is not cosimple. Conversely, if $U_\uu$ is not cosimple then there are $\vv,\ww\in U\setminus \{\uu\}$ and a linear functional $f$ with  $f(\vv) = \pm f(\ww)$ and such that $f$ vanishes on $U\setminus \{\uu,\vv,\ww\}$. This translates to a linear dependence between $\vv^*\pm \ww^*$ and $\uu^*$, so that  indeed $\vv^*\pm \ww^*$ and $\uu^*$ are parallel elements of $\double{(U^*)}$.
\end{itemize}

For part (2), suppose that  $\double{(U^*)}$ contains an element proportional to $\uu^*-\vv^*$. There are again two possibilities:
\begin{itemize}
\item The element is of the form $2 \ww^*$ for a $\ww\in U$. This implies a linear dependence of the form $\lambda_1(\uu^*-\vv^*) + \lambda_2 \ww^*=0$ in $U^*$, which translates in $U$ to a linear function $f$ that vanishes on $U\setminus \{\uu,\vv,\ww\}$ and with $f(\uu) = f(\vv)=\lambda_1$, hence $f(\uu + \vv)=0$.
Thus, $\ww$ is a coloop in $U^+_{\uu,\vv}$. Conversely, if $\ww$ is a coloop in $U^+_{\uu,\vv}$ then $\ww^*$ is parallel to $\uu^*-\vv^*$.

\item The element is of the form $\ww^* \pm \zz^*$ for $\ww,\zz\in U$. Then, we have a linear dependence  $\lambda_1(\uu^*-\vv^*) + \lambda_2 (\ww^* \pm \zz^*)=0$ in $U^*$. 
As in the previous case, this translates in $U$ to a linear function $f$ that vanishes on $U\setminus \{\uu,\vv,\ww,\zz\}$ and with $f(\uu + \vv)=f(\ww \mp \zz)=0$, so that $\ww$ and $\zz$ prove  that $U^+_{\uu,\vv}$ is not cosimple. 
Conversely, if there are $\ww$ and $\zz$ in $U^+_{\uu,\vv}$ and a functional $f$ with $f(\ww \pm \zz)=0$ and $f$ vanishing on $U\cup \{\uu+\vv\} \setminus \{\uu,\vv,\ww,\zz\}$, it translates to a linear dependence among $\uu^*-\vv^*$ and $\ww^* \mp \zz^*$.
\end{itemize}

 Part (3) is similar to (2), simply changing the sign of $\vv$.
\end{proof}

This suggests the following definition and gives the next corollary:

\begin{definition}[{\cite[
Definition 1.2]{Malikiosis2024LinExpCheckingLRC}}]
\label{def:LVP}
We say that a vector configuration $V$ has the \emph{Lonely Vector Property} if $\double{V}$ contains an element that is not proportional to any other.
\end{definition}

\begin{theorem}
\label{thm:minor-LVP}
For a  vector configuration $U$ with Gale dual $U^*$ the following are equivalent:
\begin{enumerate}
\item $U^*$ satisfies the Lonely Vector Property.

\item Some deletion $U_\uu$ or diagonal $U_{\uu, \vv}^{\pm}$ of $U$ is cosimple.
\end{enumerate}
\end{theorem}

\begin{proof}
The equivalence of the two properties is Proposition~\ref{prop:minor-LVP}. 
\end{proof}

\begin{corollary}
\label{coro:contained-sLRZ}
Let $Z$ be a cosimple zonotope with generating set $U$ and suppose $Z$ does not contain any other cosimple zonotope of the same dimension. Then, either $Z$ is an sLR zonotope or $U^*$ fails to have the Lonely Vector Property.
\qed
\end{corollary}

We now include two statements about cosimple and sLR zonotopes with many points, analogous to Theorem \ref{thm:finite-LR-new}. The statements  allow to apply induction on the dimension and prove \emph{some} upper bound for the covering radius of cosimple $d$-zonotopes, and of sLR $d$-zonotopes with the LVP property, assuming inductively some bound in one less dimension.

To this end, define $\musLR_d$ as the maximum covering radius of sLR $d$-zonotopes and $\mucosimple_d$ the same maximum over all cosimple $d$-zonotopes. Since sLR implies cosimple, $\mucosimple_d \ge \musLR_d$.

\begin{theorem}[Theorems B and C, \cite{Malikiosis2024LinExpCheckingLRC}]
\label{thm:finite-sLR-new}

Let $Z$ be a cosimple $d$-zonotope with more than $\ell^d$ lattice points, for a positive integer $\ell$. Then:
\[
\covering(Z) \le \mucosimple_{d-1} + \frac1{\ell}.
\]
Moreover, if $Z$ is an sLR zonotope with generators $U$ (that is, $|U|=d+1$),
$\pi$ is the projection from Proposition~\ref{prop:long-projection},  and the Gale transform of $(\pi(U))^*$ has the Lonely Vector Property, then 
\[
\covering(Z) \le \musLR_{d-1} + \frac1{\ell}.
\]
\end{theorem}

\begin{proof}
As in the proof of Theorem~\ref{thm:finite-LR-new-body}, if $Z$ has more than $\ell^{d}$ lattice points then the linear projection $\pi$ of Proposition~\ref{prop:long-projection} gives us
\[
\covering(Z) \leq \covering(\pi(Z) ) + \frac1{\ell}.
\]
Here, the zonotope $\pi(Z)$ is cosimple (Proposition~\ref{prop:cosimple-coloopless-projection}) of one dimension less, hence
\[
\covering(Z) \leq \covering(\pi(Z) ) + \frac1{\ell} \le 
\mucosimple_{d-1} + \frac1{\ell}.
\] 

If $Z$ is an sLR zonotope, let $U$ be its set of $d+1$ generators, and let $\pi(U)\subset \R^{d-1}$ be the set of generators of the projection. If $(\pi(U))^*$ has the LVP then  Theorem~\ref{thm:minor-LVP} guarantees that some deletion or diagonal of $\pi(U)$ is cosimple. Since this minor has $d$ generators in dimension $d-1$, it is an sLR $(d-1)$-zonotope, hence its covering radius is $\covering(\pi(Z) )\le \musLR_{d-1}$ and
\[
\covering(Z) \leq \covering(\pi(Z) ) + \frac1{\ell} \le 
\musLR_{d-1} + \frac1{\ell}.
\qedhere
\] 
\end{proof}

Observe that, in the conditions of Theorem~\ref{thm:finite-sLR-new}, if $\covering(\pi(Z)) \le \frac{d-1}{d+1}$ for such a projection 
and $\ell\ge \binom{d}2$ then $\covering(Z) \le \frac{d}{d+2}$. Indeed:
\[
\covering(Z) \leq \covering(\pi(Z) ) + \frac1{\binom{d+2}2} \le 
\frac{d-1}{d+1}  + \frac1{\binom{d+2}2}   = \frac{d}{d+2}.
\]  
Via this remark, Theorem~\ref{thm:finite-sLR-new} implies Theorems B and C in  \cite{Malikiosis2024LinExpCheckingLRC}.

\subsection{Vector configurations without the Lonely Vector Property}
\label{sec:noLVP}

To construct counterexamples to LVP, we first rephrase it in a more symmetric way. In the following statement, for a configuration $\overline S$ we denote by $\overline S+\overline S$ the configuration of pairwise sums of elements of $\overline S$. We consider it a configuration of size $\binom{|\overline S|+1}{2}$ since we allow the sum $\uu+\uu$ of an element with itself, and since we identify $\uu+\vv = \vv +\uu$, but we do not identify other sums that may happen to give the same result.

\begin{lemma}
\label{lemma:LVP-symmetric}
Let $S$ be a vector configuration containing neither $\oo$ nor two elements with $\uu\pm\vv=\oo$. Let $\overline S := S \cup (-S) \cup\{\oo\}$ (of size $2|S|+1$).
Then, the following are equivalent:
\begin{enumerate}
\item $S$ has the Lonely Vector Property.
\item In the configuration $\overline S + \overline S$ (of size $(|S|+1)(2|S|+1)$) some vector is not a positive multiple of any other.
\end{enumerate}
\end{lemma}

\begin{proof}
Suppose $S$ has the Lonely Vector Property, so there is either an $\uu$ or an $\uu\pm \vv$ that is not proportional to any other element of $\double{S}$. Then the element $2\uu$ or $\uu\pm \vv$ is not positively proportional to any other element of  $\overline S + \overline S$. (Observe we need $\uu\pm \vv\ne \oo$, since otherwise $\uu\pm \vv$ is proportional to $\uu' - \uu'=\oo$). 

Conversely, suppose that there is a non-zero element $\zz = \pm \uu \pm \vv \in \overline S + \overline S$ (with $\uu,\vv \in S$ and perhaps equal to one another) that is not positively proportional to any other. There are three cases:
\begin{itemize}
\item If $\uu=\vv$ then $\zz=\pm (\uu+\uu)$ (since $\zz=\pm (\uu-\uu)$ gives $\zz=\oo$, and $\oo$ appears multiple times in $\overline S + \overline S$). Then $\tfrac12\zz =\uu$ is in $\double{S}$ and is not proportional to any other element of $\double{S}$.
\item If $\uu\ne\vv$ and $\zz = \pm(\uu+\vv)$ then $\uu+\vv\in \double{S}$ is not proportional to any other element of $\double{S}$.
\item If $\uu\ne\vv$ and $\zz = \pm(\uu-\vv)$ then one of $\uu-\vv$ or $\vv-\uu$ is in $\double{S}$ and is not proportional to any other element of $\double{S}$.
\qedhere
\end{itemize}
\end{proof}

That is, finding a counterexample to LVP is the same as finding a centrally symmetric configuration $\overline S$ such that every vector of $\overline S + \overline S$ is positively proportional to some other.
With this we can give our first family of examples:

\begin{proposition}
\label{prop:LVP-rectangle}
Let $\overline S$ be the set of lattice points in the rectangle $[-a,a]\times[-b,b]$ for some $a,b\in \Z_{>0}$.
Let $S$ be obtained by removing the vector $\oo$ and taking only one copy of each $\pm \vv$ in $\overline S$. Then, the following are equivalent:
\begin{enumerate}
\item $S$ satisfies the LVP.
\item At least one of the vectors $(2a-1,2b)$ or $(2a,2b-1)$ is primitive.
\end{enumerate}
\end{proposition}

\begin{proof}
Let us call a vector  $\pp \in \overline S+\overline S$ \emph{lonely} if $\overline S+\overline S$ contains no vector paralell to $\pp$ other than $-\pp$. Observe that in order to be lonely a vector needs to have multiplicity one (i.e., not to be repeated in $\overline S+\overline S$, which is a multiset). So, let us look at what vectors can have multiplicity one in $\overline S+\overline S$.

If we distinguished between the sums $\uu+\vv$ and $\vv+\uu$, the multiplicity of each vector $\pp\in \overline{S} + \overline{S}$  would simply be  the cardinality of $\overline{S} \cap (\pp - \overline{S})$. For $\pp=(x,y)$ and $\overline S=[-2a,2a]\times[-2b,2b]$  in $\overline{S}+\overline{S}$ this equals $(2a+1-|x|)(2b+1-|y|)$. Now, the fact that we identify $\uu+\vv$ and $\vv+\uu$  divides the count by two, except that if both $(x,y)$ are even we have to adjust for the sum $(x,y)=(x/2,y/2)+(x/2,y/2)$ that was not counted twice.
Thus, the multiplicity of $(x,y)$ when we identify $\uu+\vv$ and $\vv+\uu$ is:
    \[
\begin{cases}
\dfrac{(2a+1-|x|)(2b+1-|y|)}{2} +\frac12& \text{if } x, y \text{ are both even}, \\[2ex]
\dfrac{(2a+1-|x|)(2b+1-|y|)}{2} & \text{otherwise}.
\end{cases}
\]

Hence, the only vectors of multiplicity one  in $\overline S +\overline S$  are $(\pm2a, \pm2b)$, $(\pm(2a-1),  \pm2b)$ and $( 2a,(2b-1))$. 
The statement then follows from the fact that, for these particular vectors, being lonely is equivalent to being primitive.
\end{proof}

\begin{corollary}
\label{coro:LVP-rectangle}
Letting $a=3$ and $b=5$ in the previous statement provides a set $S$ of $(7\cdot 11-1)/2 = 38$ integer points with no zero elements, no element equal or opposite to another one, and without the Lonely Vector Property.
\end{corollary}

\begin{proof}
$(2a-1,2b)=(5,10)$ and $(2a,2b-1)=(6,9)$ are both non-primitive.
\end{proof}

Figure~\ref{fig:noLVP12} shows smaller examples. The pictures represent centrally symmetric sets $\overline S$ that fail to satisfy condition (ii) of Lemma~\ref{lemma:LVP-symmetric}. For added symmetry the first picture is in a regular triangular grid, but the example can obviously be linearly transformed to an integer one. Hence:

\begin{corollary}
\label{coro:noLVP12}
There is a set $S$ of $12$ lattice vectors, none of them zero and no two equal or opposite, that fails to have the Lonely Vector Property.
\end{corollary}

\begin{proof}
By inspection of the left part of Figure~\ref{fig:noLVP12}. Left to the reader.
\end{proof}

\begin{figure}[h]
\raisebox{.5cm}{\includegraphics[scale=0.25]{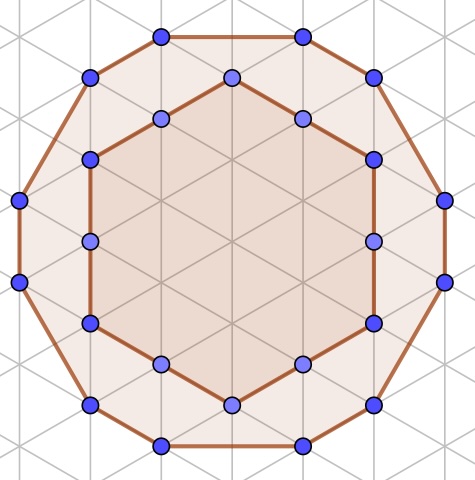}}\quad
\includegraphics[scale=0.25]{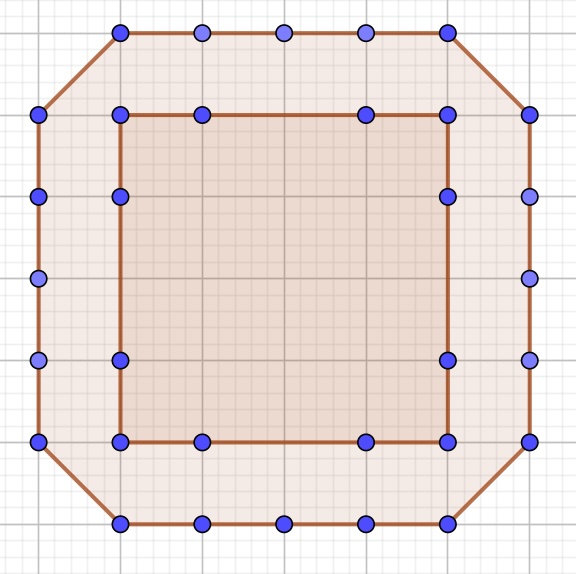}
\caption{Left: the symmetrized configuration $\overline S$ of a counterexample $S$ to LVP with $|S|=12$. Right: same with $|S|=16$}
\label{fig:noLVP12}
\end{figure}

\subsection{All cosimple zonotopes properly contain zonotopes of width at least three}
\label{sec:width3}

We know that some cosimple zonotopes do not contain any cosimple diagonal or deletion, because that is equivalent to the Gale dual satisfying the Lonely Vector Property, and some integer configurations do not have that property. 
We also know that cosimple zonotopes are almost the same as lattice zonotopes of width at least three (Lemma~\ref{prop:cowidth-sLR}).
It may then be surprising that, as we now show in Corollary~\ref{coro:contained-wide}, every cosimple zonotope $Z$
contains a diagonal $Z'$ of width at least three, when width in $Z'$ is measured with respect to the lattice generated by $Z$ (and hence $Z'$ is not necesarilly cosimple).
Observe that this includes sLR zonotopes, a case in which all diagonals are necessarily parallelepipeds.

\begin{lemma}
\label{lemma:contained-wide}
Let $Z$ be the zonotope generated by a configuration $U=\{\uu_1,\dots,\uu_n\}$ and assume there is 
a linear dependence $\sum_{i=1}^n v_i \uu_i=\oo$ with $0 < v_1 < \ldots < v_n$. Then, $\zono[U_{n-1,n}^-]$ has width $\ge 3$ with respect to the lattice generated by $U$.
\end{lemma}

\begin{proof}
Let $Z':=Z_{n-1,n}^-$ and let $f\not\equiv 0$ be an integer functional. We need to show that  the integer sum
\begin{align}
\label{eq:sum-cosimple}
 \width(Z',f) = \sum_{i=1}^{n-2} |f(\uu_i)| + |f(\uu_{n-1}-\uu_n)| 
\end{align}
is greater than two. 
If that is not the case then at most two of the values in the sum are non-zero, and they add up to at most two. We consider four cases, depending on whether one or two values are non-zero, and whether $|f(\uu_{n-1}-\uu_n)| $ is one of them.
In all the cases we use that
\[
\sum_{i=1}^n v_i f(\uu_i) = f\left(\sum_{i=1}^n v_i \uu_i \right)=0,
\]
together with the hypothesis that $0<v_1 < \dots < v_{n-1} < v_n$.

\begin{enumerate}
\item If the only non-zero summand in \eqref{eq:sum-cosimple} is $|f(\uu_i)|$ for $i\in\{1,\dots,n-2\}$ we get the contradiction $v_i \ne 0 \ne f(\uu_i)$ but $v_i f(\uu_i) = 0$.

\item If the only non-zero summand in \eqref{eq:sum-cosimple} is $|f(\uu_{n-1}-\uu_n)|$ we have that
\[
v_{n-1} f(\uu_{n-1}) + v_{n} f(\uu_{n}) =0.
\]
This together with $0< v_{n-1} < v_n$ implies that both $f(\uu_{n-1})$ and $f(\uu_{n})$ are non-zero. Then
\[
f(\uu_{n}-\uu_{n-1}) = f(\uu_{n})-f(\uu_{n-1}) =  \frac{v_n + v_{n-1}} {v_{n-1}} f(\uu_n),
\]
which implies $|f(\uu_{n}-\uu_{n-1})| > 2$ since $ \frac{v_n + v_{n-1}} {v_{n-1}} >2$.

\item If there are two non-zero summands but none of them is $|f(\uu_{n-1}-\uu_n)|$, let $f(\uu_i)$ and $f(\uu_j)$ be the non-zero ones, with $i<j<n-1$ and hence $v_i <v_j$.
Then
\[
v_{i} f(\uu_{i}) + v_{j} f(\uu_{j}) =0
\]
implies that $|f(\uu_{i})| \ne |f(\uu_{j})|$ and they are both positive integers, so their sum is at least three.

\item Finally, if there are two non-zero summands and they are $|f(\uu_{n-1}-\uu_n)|$ and a certain $|f(\uu_{i})|$, $i<n-1$, we want to show that these summands cannot both be equal to $1$.
The equation
\[
v_i f(\uu_{i}) + v_{n-1}f(\uu_{n-1}) + v_n f(\uu_{n})=0
\]
 implies
\[
 -f(\uu_{n})= \frac{  f(\uu_{i}) v_{i} + ( f(\uu_{n-1})-f(\uu_{n})) v_{n-1}}{v_{n-1}+v_n}.
\]
Since $0 < v_i < v_{n-1}< v_n$, the assumption $|f(\uu_{n-1}-\uu_n)|= |f(\uu_{i})| = 1$ leads to 
$ |f(\uu_{n})|<1$, hence $f(\uu_{n})=0$ and
\[
v_i f(\uu_{i}) + v_{n-1}f(\uu_{n-1}) =0.
\]
As in the previous case, we then have $|f(\uu_{i})| \ne 0 \ne |f(\uu_{n-1})|$ so their sum is at least three.
\qedhere\end{enumerate}
\end{proof}

\begin{corollary}
\label{coro:contained-wide}
Every cosimple zonotope $Z$  contains a diagonal $Z'$ that has width at least three with respect to the lattice of $Z$.
\end{corollary}

\begin{proof}
By definition of cosimple, there is a linear dependence  $\sum_{i=1}^n v_i \uu_i=\oo$ among the generators of $Z$  such that all the $v_i$ are non-zero and have different absolute values. By reordering the generators and changing them to their opposites if needed, there is no loss of generality in assuming 
$0 < v_1 < \ldots < v_n$. Then, Lemma~\ref{lemma:contained-wide} gives the result.
\end{proof}

The following example shows that the diagonal constructed in the proof of Lemma~\ref{lemma:contained-wide} may be not cosimple, even having width at least three:

\begin{example}
Let $Z$ be the $2$-zonotope with generators:
\[
\uu_1 = (1,2), \uu_2 = (2,1), \uu_3 = (1,-2),  \uu_4 = (-2,1).
\]
It is cosimple because, in particular, there is the following linear dependence among the generators (see Definition~\ref{def:co}):
\[
\uu_1 + 4 \uu_2 + 7 \uu_3 + 8 \uu_4 =\oo,
\]
and the generators are ordered as in the hypothesis of Lemma~\ref{lemma:contained-wide}.

The subzonotope $Z':=\zono[U_{3,4}]^-\subset Z$,  with generators $\uu_1 = (1,2)$, $\uu_2 = (2,1)$ and  $\uu_3' = \uu_3-\uu_4 = (3,-3)$ is NOT cosimple, since the only non-trivial dependence among its generators is $3 \uu_1 - 3 \uu_2 + \uu'_3 = \oo$, in which $\uu_1$ and $\uu_2$ have coefficients of equal absolute value.
$Z'$ has width $\ge 3$ with respect to $\Z^2$, but its generators span the lattice $x + y \equiv 0 \pmod{3} \subsetneq \Z^2$ and, with respect to this lattice, $\Z'$ has width $2$.
\end{example}\vspace{.5cm}

\section{Counterexamples to the shifted LRC}
\label{sec:counter}

Our counterexamples of the sLRC (including those in Theorem~\ref{thm:12345}) where obtained and certified via an algorithm to compute the shifted loneliness gap $\sgap(\vv)$ of an input vector $\vv \in \Z^n_{>0}$. 
Recall that this is equivalent to computing the covering radius $\covering(Z)$ of the associated LR zonotope, via Proposition~\ref{prop:gamma-kappa-mu}. A general-purpose algorithm for the covering radius of an arbitrary rational polytope was described in~\cite{ACS4slrz}, and used by the authors of that paper to prove the shifted LRC for $n=4$. 
The algorithm that we now propose is more specific, based on regarding LR zonotopes as \emph{polytropes}, objects that are polytopes both in the usual sense and in the tropical-geometric sense. Although we do not assume the reader to have any background on them, for their general theory see, e.g.,  Sections 3.4 and 6.5 in~\cite{Joswig-etc}. 

Both the code we used and the full list of counterexamples that we have obtained are available at \cite{code_git}.

\subsection{The algorithm}
\label{sec:algorithm}

The main routine in the algorithm has as input both a primitive vector $\vv \in \Z^n_{>0}$ and a candidate value $\gamma \in(0, 1/2)$, and decides between the  three possibilities
\[
\sgap(\vv) > \gamma,
\quad
\sgap(\vv) = \gamma
\quad \text{ or } \quad
\sgap(\vv) < \gamma.
\]
To find the exact loneliness gap $\sgap(\vv)$ we combine this test with a binary search on the Stern-Brocot tree of rational numbers in the interval $(0, 1/2)$. The search finishes thanks to the fact that the denominator of the covering radius of $Z$ can be bounded in terms of the entries of $\vv$. See, e.g., \cite{ACS4slrz} for that part.
In what follows we only describe the decision algorithm that compares an input $\gamma$ with $\sgap(\vv)$.

Recall that
\[
\sgap(\vv):= \min_{\ss\in [0,1]^n}\,\gap(\vv; \ss).
\]
That is, we have the configuration space of possible values $\ss$, which are important only modulo the integers.
For convenience, we change from the parameters $\ss$ to new parameters $\xx=(x_1,\dots,x_n)\in \R^n_{\ge 0}$ defined by $x_i=-s_i /v_i$. 
Since the initial $\ss$ is only important modulo the integers, we assume without loss of generality that 
 $s_i\in[-1,0]$, hence $0\leq x_i \leq 1/v_i$ for all $i$. If  $v_i$ is interpreted as a velocity and $s_i$ as a starting position,  $x_i$ equals the first time when the $i$-th runner hits the origin. 

Since, moreover, there is no loss of generality in assuming that $x_1=0$, our initial configuration space is 
 \begin{align}
\label{eq:starting}
X=\{0\} \times \prod_{i=2}^n [0, 1/v_i].
\end{align} 
(For added performance we can further restrict this space exploiting symmetry; see Lemma~\ref{lem:startingdomain} below).
\bigskip

The inequality $\sgap(\vv) > \gamma$ is equivalent to the following:  for each starting value $\xx\in X$, a time $t$ exists when all  the runners are at distance strictly greater than $\gamma$ from the integers. In other words, a $t\in (0,\infty)$ that satisfies 
\begin{align}
\label{eq:certificate}
(t-x_i) v_i \in(k_i+\gamma, (k_i+1)-\gamma) \quad \forall i,
\end{align}
for some choice of integers $k_1,\dots, k_n\in \Z$. (Each $k_i$ represents which loop the $i$-th runner is in when loneliness is achieved). 
The following lemma describes the subset of $X$ for which such a $t$ exists for each choice of $k_1,\dots, k_n\in \Z$, and shows that without loss of generality we can assume $k_i\in [-1,v_i-1]$.

\begin{lemma}
\label{lemma:polytrope}
Let $\xx=(x_1,\dots, x_n) \in X$. Then, 
\begin{enumerate}
\item A time $t\in \R$ and a $\kk=(k_1,\dots, k_n)\in \Z^n$ satisfying \eqref{eq:certificate} exist for $\xx$ if, and only if  the following inequalities, involving only $\xx$ and $\kk$ are met: 
\begin{align}
\label{eq:polytrope}
 x_j-x_i <  \frac{(k_i-\gamma+1)}{v_i} -  \frac{(k_j+\gamma)}{v_j}, \quad \forall i,j\in [n].
 \end{align}

\item If this happens for some $\kk \in \Z^n$ then it also happens for one having  $-1\leq k_i \leq v_i-1$ for all $i=1,\dots,n$.
\end{enumerate}
\end{lemma}

\begin{proof}
Equations  \eqref{eq:certificate} are equivalent to
\begin{align}
\label{eq:certificate-t}
t \in \left(\frac{k_i+\gamma}{v_i} + x_i, \frac{(k_i+1)-\gamma}{v_i}+x_i \right) ,
\end{align}
For fixed $\kk$,  a $t$ satisfying all these inequalities exists if and only if every lower bound is smaller than every upper bound. That is, if and only if
\[
 \frac{(k_j+\gamma)}{v_j} +x_j<  \frac{(k_i+1-\gamma)}{v_i} + x_i \ \ \forall i,j\in [n],
\]
which is equivalent to the statement.
(The case where $i=j$ is never an issue, since $\gamma<1/2$.) This proves part (1).

For part (2), observe that adding an integer $t_0$ to $t$ and the vector $t_0 \vv$ to $\kk$ leaves equations~\eqref{eq:certificate} invariant. Hence, if~\eqref{eq:certificate} is feasible for the given $\xx$, it has a solution with $t\in [0,1]$. 
On the other hand,  \eqref{eq:certificate-t} together with $x_i\in [0,1/v_i]$ and $\gamma\ge 0$, implies
\[
t \in \left(\frac{k_i}{v_i} , \frac{k_i+2}{v_i}  \right).
\]
Since $k_i$ is an integer, for this to  intersect $[0,1]$ we need $k_i \in [-1,v_i-1]$.
\end{proof}

Our algorithm is based on (trying to) cover $X$ by the polytopes defined by the inequalities \eqref{eq:polytrope} in part (1) of Lemma~\ref{lemma:polytrope}; part (2) tells us that 
the number of  polytopes that the algorithm needs to try is finite, bounded by $\prod_{i}(v_i +1)$.
Let us give a name to these polytopes, both in the version of the lemma and in a 
weak version of loneliness, in which runners are required to be in the \emph{closed} interval $[k_i+\gamma, (k_i+1)-\gamma]$. 
By continuity and compactness, existence of a $t$ for given $k_1,\dots, k_n\in \Z$ in this weak version is equivalent to feasibility of the closed version of the system~\eqref{eq:polytrope}.

\begin{definition}
For each $\kk \in \Z^n$ and $\gamma\in (0,1/2)$ we call the set of $\xx\in \R^n$ satisfying  \eqref{eq:polytrope} (resp. its closure) the open (resp. closed) \emph{certificate polytrope} for $\gamma$ in round $\kk$. We denote $T_{\kk,\gamma}$ the open one and $\overline T_{\kk,\gamma}$ its closure.
\end{definition}

\newcommand{\dsone}{\mathbf 1}

The reason for this name is that polytopes defined by inequalities of the form $x_i-x_j \le b_{ij}$ are called \emph{polytropes} in tropical geometry (in Coxeter combinatorics they are called \emph{alcoved polytopes}). They contain $\dsone \R $ in their linearity space so we consider them as living in the quotient space $\R^n/\dsone\R$, where they are bounded (if all the inequalities $x_i-x_j \le b_{ij}$ are present in their definition, as is our case).

\begin{corollary}
Let $\vv \in \Z_{>0}^n$ and let $\gamma\in (0, 1/2)$.
\label{coro:polytrope}
\begin{enumerate}
\item $\sgap(\vv) > \gamma$ if, and only if, 
\begin{align*}
& X \subset \bigcup_{k_i\in [-1,v_i-1]}  T_{\kk,\gamma}.&\quad
\end{align*}
\item $\sgap(\vv) \geq \gamma$ if, and only if, 
\begin{align*}
 &X  \subset \bigcup_{k_i\in [-1,v_i-1]}  \overline T_{\kk,\gamma}.& \qed
\end{align*}
\end{enumerate}
\end{corollary}

The (closed) certificate polytropes $\overline T_{\kk,\gamma}$ are nothing but affine transforms of the LR zonotope corresponding to the velocity vector $\vv$. To see this, first observe that all the $T_{\kk,\gamma}$ are translations of one another: adding $\epsilon$ to a particular $k_i$ is equivalent to subtracting $\epsilon/v_i$ from $x_i$. Hence, we focus for now on the polytrope for $\kk=\oo$, as the others are equivalent via lattice translation.

Now write (the closed version of) equations \eqref{eq:polytrope} back in the coordinates $s_i = -x_i v_i$ that represent initial positions. The equations become
\begin{align}
\label{eq:closed_polytrope}
 \frac{s_i -  (1-\gamma)}{v_i}  \leq  \frac{s_j-\gamma}{v_j}, \quad \forall i,j\in [n].
 \end{align}
These are the same as the defining inequalities of
$[\gamma,1-\gamma]^n + \R\vv$. 
Hence:

\begin{proposition}
\label{prop:T_vs_Z}
The linear isomorphism 
\begin{align*}
\pi: \R^n /\dsone \R  \quad&\stackrel{\cong}{\longrightarrow} \quad \R^n/ \R\vv \cong \R^{n-1} \\
\xx =(x_i)_i \quad&\mapsto \quad \ss:= - (x_i v_i)_i
\end{align*}
sends 
$\overline T_{\oo,\gamma}/ \dsone \R$ to the projection  of the cube $[\gamma,1-\gamma]^n / \R\vv $.
\qed
\end{proposition}

\begin{remark}
Letting $\gamma=0$ we get $\pi(\overline T_{\oo,0}/ \R\dsone ) = [0,1]^n/\R\vv$ which, by Proposition~\ref{prop:isomorphism} is isomorphic to the LR zonotope $\zono[\vv]$. 

Similarly, $\pi(\overline T_{\oo,\gamma}/ \dsone \R )$ equals $\zono[\vv]$ contracted by a factor of $\mu:=1-2\gamma$ from its center and
the family $\{\pi(\overline T_{\kk,\gamma}/ \dsone \R )\}_{\kk\in \Z^n}$ are its translations under the lattice $\pi(\Z^n)$.
This makes part (2) of Corollary~\ref{coro:polytrope}  a rephrasing of the formula 
\[
\covering(Z) =1 - 2\,\sgap(\vv),
\]
from Proposition~\ref{prop:gamma-kappa-mu}.
Indeed, $\covering(Z)\le 1-2\gamma$
is equivalent to being able to cover a fundamental domain, the polytrope $X=\pi([0,1]^n)$, by lattice-translated copies of $(1-2\gamma) Z$. 
\end{remark}

The reason for the change of coordinates $\ss$ to the coordinates $\xx$ is precisely that it allows us to take algorithmic advantage of some particular properties of polytropes.
The main such property that our algorithm uses is that the exterior region of a certificate polytrope $\overline T_{\kk,\gamma}$ can be naturally subdivided into polytropal (unbounded) regions as follows. Consider $\overline T_{\kk,\gamma}$ as defined by the $n(n-1)$ inequalities $x_i-x_j \le b_{ij}$. That is, let 
\[
b_{ij} =  \frac{(k_i-\gamma+1)}{v_i} -  \frac{(k_j+\gamma)}{v_j}
\]
 be the maximum over $P$ of the functional $x_i-x_j$, as in Eq.~\eqref{eq:polytrope}. 

\begin{definition}
Let $\overline T_{\kk,\gamma} = \{\xx \in \R^n/\dsone \R  : x_i-x_j \le b_{ij}\}$ be a certificate polytrope.
The \emph{exterior polytropal region} of $\overline T_{\kk,\gamma}$ corresponding to indices $(i,j)\in [n]^2$, $i\ne j$, is the unbounded polyhedron
\[
\big\{\xx \in \R^n/\dsone \R  : x_k-x_i \le b_{kj}- b_{ij}, \ x_k-x_j \le b_{ij}- b_{ik} , \ k\in[n]\setminus \{i,j\}\big\}.
\]
We call the polyhedral decomposition of $\R^n$ into $P$ and these $n(n-1)$ regions the \emph{polytropal tiling} induced by $\overline T_{\kk,\gamma}$.
See Figure~\ref{fig:tiling}.
\end{definition}

\begin{remark}
The polytropal tiling induced by a certificate polytrope $\overline T_{\kk,\gamma}$ can be understood as a Voronoi diagram in two ways:
\begin{itemize}
\item It is (the intersection of $\R^n/\dsone \R \setminus \overline T_{\kk,\gamma}$ with) the usual Voronoi diagram of the facets of $\overline T_{\kk,\gamma}$. Indeed, every two adjacent facets (facets sharing a ridge)  of $\overline T_{\kk,\gamma}$ are either of  the form $\{(i,j), (k,j)\}$ or of the form  $\{(j,i), (j,k)\}$. The bisector of such facets is contained in the hyperplane  $\{x_k-x_i = b_{kj}- b_{ij}\}$ or   $\{x_k-x_i = b_{ji}- b_{jk}\}$, respectively.

\item Each point $\yy=(y_i)_{i=1}^n$ in the exterior of $\overline T_{\kk,\gamma}$ is assigned to the pair $(i,j)$ maximizing the amount $y_i-y_j - b_{ij}$ by which $\yy$ violates the corresponding facet-inequality of $\overline T_{\kk,\gamma}$. Hence, the tiling is the \emph{farthest site Voronoi diagram} of the collection of half-spaces $\{x_i-x_j \le b_{ij}\}_{(i,j)\in [n]^2}$. 
\end{itemize}
\end{remark}

\begin{figure}[htb]
\raisebox{.2cm}{\includegraphics[height=3.8cm]{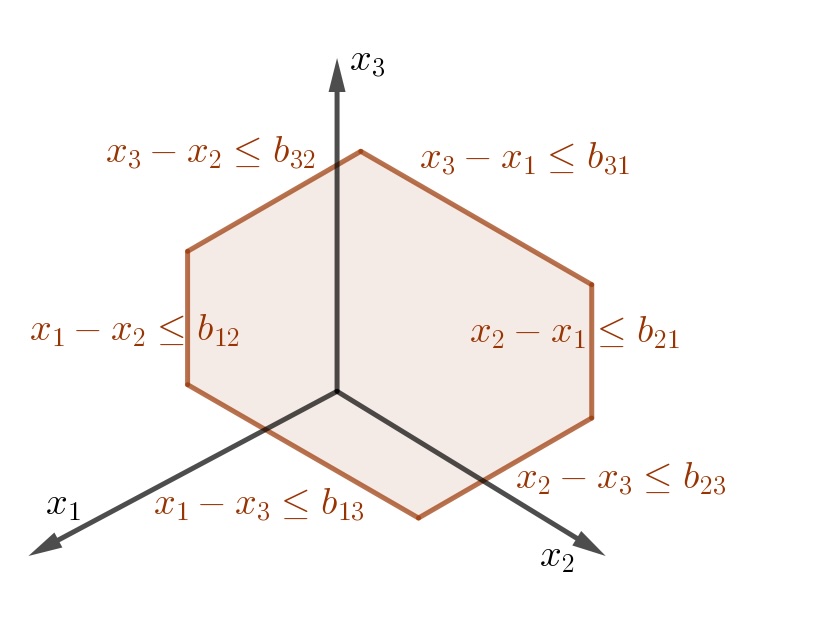}}\quad
\includegraphics[height=4cm]{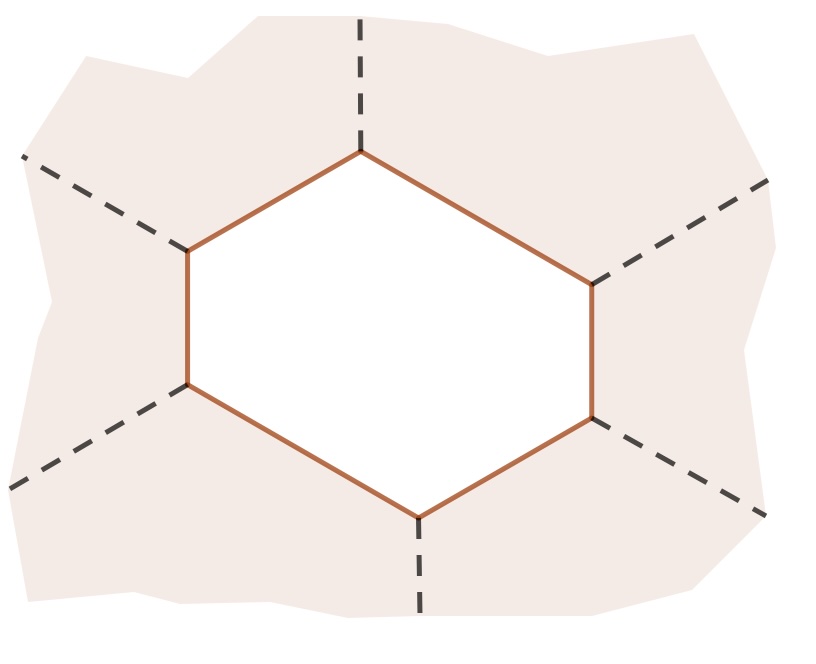}
\caption{Left: A certificate polytrope $\overline T_{\kk,\gamma}$ in $\R^3/\dsone\R^3$. Right: the exterior of $\overline T_{\kk,\gamma}$ is subdivided into  the polytropal regions of the facets of $\overline T_{\kk,\gamma}$.}
\label{fig:tiling}
\end{figure}

Observe that $X/\dsone\R$ is a polytrope too. Indeed, since $x_1=0$ over $X$, the definition of $X$ in \eqref{eq:starting} is equivalent to
\[
X= \{x_1=0\} \cap \{ \xx\in \R^n : 0 \le x_i - x_1 \le 1/v_i, \quad i=2,\dots, n\} .
\]

We now describe our algorithm, in the version using the open certificates (that is, the one that tests part (1) of Corollary~\ref{coro:polytrope}).
See Remark~\ref{rem:closed_certificates} for the main differences that need to be taken into account for the version with closed certificates.

The algorithm consists in computing several certificate polytropes one after another and keep track of the uncovered areas of the parameter space $X$.
Throughout the algorithm, the uncovered region is updated as a list $L$ of closed non-empty polytropes, 
which  at the beginning consists of the single polytrope $X$. In each iteration, while $L$ is not empty, the algorithm selects a polytrope from it, call it $Y$, computes a point $\xx$ in the relative interior of $Y$, and checks whether a certificate polytrope $T_{\kk,\gamma}$ containing $\xx$ exists.
Three things can happen:
\begin{enumerate}
\item[(a)] Such a $T_{\kk,\gamma}$ does not exist (that is,  the system~\eqref{eq:polytrope} is infeasible for every $\kk$ in $\prod_{i=1}^{n} \{-1,\dots, v_i-1\}$). Then we have found starting positions $\ss = -({v_1}{x_1}, \dots,{v_n}{x_n})$  that prove $\sgap(\vv) \le \gamma$.

\item[(b)] A polytrope $T_{\kk,\gamma}$ containing $\xx$ exists, and it covers $Y$. Then we remove $Y$ from the list $L$. If the list becomes empty then Corollary~\ref{coro:polytrope} implies that $\sgap(\vv) > \gamma$; if the list is still not empty we iterate.

\item[(c)] If the new certificate polytrope $T_{\kk,\gamma}$ containing $\xx$ covers only part of $Y$ then we remove $Y$ from the list $L$, use the polytropal tiling of $T_{\kk,\gamma}$ to split $Y\setminus T_{\kk,\gamma}$ into polytropes, and insert these polytropes in $L$
 (see again Figure~\ref{fig:tiling}). 
Observe that we only want to remove the \emph{open} polytrope $Y$, so the new polytropes inserted are closed.
\end{enumerate}

In each iteration, both $\xx$ and $\kk$ admit a priori several solutions. We have picked fast heuristics that satisfy our geometric intuition of the problem; we describe them in Section~\ref{sec:heuristics}.

\begin{lemma}
This algorithm terminates. That is, the third case above can happen only a finite number of times.
\end{lemma}

\begin{proof}
Every polytrope that appears along the algorithm is a union of cells in the  common refinement of the polytropal tilings of all the certificate zonotopes $T_{\kk,\gamma}$ that can potentially be used during the algorithm. There is a finite number of such polytropes, since 
$-1\leq k_i \leq v_i-1$ for every $i$. Since every time that we go through (c) at least one cell of the arrangement is removed from the uncovered region, case (c) can occur only a finite number of times.

It may happen that the $Y$ in a particular iteration, and hence the part $Y\setminus T_{\kk,\gamma}$ of $X$ removed in this iteration, is lower-dimensional. 
But this does not invalidate the argument; every iteration removes from the uncovered space at least one relatively open cell of the finite polytropal arrangement obtained with all the $T_{\kk,\gamma}$ and their polytropal tilings, and the number of such cells is finite.
\end{proof}

\begin{remark}[The version with closed certificates]
\label{rem:closed_certificates}
In this version the uncovered space is no longer closed. But we can take advantage of the fact that, if the condition in part (2) of Corollary~\ref{coro:polytrope} is not met, then part of $X$ that cannot be covered is full-dimensional. Hence, rather than keeping track of what lower-dimensional parts of the boundary of each polytrope in $L$ are already covered, we still keep all polytropes stored as if they were closed in the understanding that:

\begin{itemize}
\item If at some point a polytrope to be added to $L$ is not full-dimensional then we simply do not add it. This may imply we are neglecting some lower-dimensional pieces of $X$ that have not yet been covered, but it does not affect the correctness of the output.

\item We need the certificate polytrope $\overline T_{\kk,\gamma}$ to have a full-dimensional intersection with the  $Y\in L$ currently being processed, in order to guarantee that we make progress at every iteration. 
Since  $\overline T_{\kk,\gamma}$ is itself full-dimensional (see, e.g., Proposition~\ref{prop:T_vs_Z}), it is sufficient to choose the new point $\xx$ at each iteration in \emph{the interior} of $Y$.
Our choice of $\xx$ satisfies this property, as we show in the next section.
\end{itemize}

The proof that the algorithm terminates is the same as in the version with open certificates, except now we only consider the full-dimensional cells in the joint polytropal arrangement.
\end{remark}

In practice, the source code determines the relationship between $\sgap(\vv)$ and a guessed value $\gamma$ in a single pass. It starts with the open-certificate version of the algorithm. If it finds a point $\xx$ that lies in no open certificate polytrope but does lie in some closed certificate polytrope, then it continues with the closed-certificate version. At that moment, the cells change from closed polytropes to open ones. This causes no problem, since only a subset of measure zero is removed. In particular, the complement of the closed certificate polytropes is open, so this change cannot eliminate all exterior points.

\subsection{Three technical details of the algorithm}
\label{sec:heuristics}

We here elaborate on some aspects of the algorithm that were glossed over in the previous explanation.

\subsubsection*{1. How to find an interior point $\xx$ in the polytrope $Y\in L$.} 
Finding a (relative) interior point of a polytope can be done with general purpose algorithms, but we need a fast way of doing it. For this we take advantage of the fact that our polytopes are polytropes.

If $(b_{ij})_{i,j\in[n]}$ is the vector of right-hand sides in the definition of our polytrope $Y$ (that is, $b_{ij}$ is the maximum of the functional $x_j-x_i$ on $Y$) we call $i$-th \emph{min-vertex} of $Y$ the point
\[
(b_{ij})_{j\in [n]} \in \R^n. 
\]
The reason for this name is that $Y$ equals the tropical convex hull of these points with respect to the min-plus algebra. (See, e.g.,~\cite[Theorem 6.37]{Joswig-etc}).

The min-vertices cannot all lie in the same proper face of $Y$: The $i$-th and $j$-th min-vertices have $x_i-x_j=-b_{ij}$ and $x_i-x_j=b_{ji}$ respectively so, if all the min-vertices lie in a facet $F_{ij}$ of $Y$ we have $b_{ij}+b_{ji}=0$ and hence $Y= F_{ij}$.
Hence, any positive barycentric combination of the min-tropical vertices of $Y$ lies in the relative interior of $Y$.
We pick as new point $\xx$ the average (i.e. the barycenter) of the min-tropical vertices of $Y$. 

If we wanted to minimize the total number of certificates needed for the overall algorithm, it would be sensible to design a choice of $\xx$  more centered in $Y$, since this should produce a more balanced partition of $Y\setminus T_{\kk,\gamma}$.
However, experimentally the performance of the algorithm was more sensitive to the speed at which $\xx$ was computed than it was to the quality of $\xx$. 

\subsubsection*{2. How to find a $T_{\kk,\gamma}$ containing  a given $\xx\in X$.} 
For given values $\xx,\vv$ we want to check whether there is a $\kk \in \Z^n$ such that the origin is lonely within the loop $k_i$ of each runner $i$. If the answer is yes then the algorithm removes the corresponding certificate polytrope $T_{\kk,\gamma}$ from the uncovered space; if the answer is no then we are in case (a) of the algorithm, so we have proved $\sgap(\vv) \leq \gamma$.

Each choice of $i$ and $k_i$ corresponds to a closed interval 
\[
\left[x_i +\frac{(k_i-\gamma)}{v_i} , x_i + \frac{(k_i+\gamma)}{v_i} \right]
\]
of time where the origin is \emph{not} lonely because runner $i$ is close to it. (Compare to \eqref{eq:certificate-t}, which describes the intervals when the $i$-th runner is far from the origin).
Lemma~\ref{lemma:polytrope} implies that we only need to look at the $\kk$ with  $-1\leq k_i < v_i$, so the question is whether the corresponding $\sum v_i + n$ intervals cover $[0,1]$.
This can be checked in time $O( \sum v_i\log(\sum v_i))$ via the classical greedy interval cover algorithm~\cite[Chapter~4, Problem~3]{Erickson-algorithms}\cite[Section~3.4.1]{Halim-CP4}.

Our implementation does a minor refinement over this classic idea by exploiting that different intervals of the same runner do not intersect, so the priority queue that the algorithm maintains contains only $n$ entries at the same time, improving the running time to $O(\sum v_i \log n)$. 

When the intervals do not cover, out of all the possible $\kk$, we pick the one corresponding to the longest uncovered interval in the complement of the union; that is, the one that keeps the origin lonely for the longest time. This  heuristic is sensible, and it performs well experimentally. Another alternative would be finding the $\kk$ that gives the origin the largest loneliness gap.

\subsubsection*{3. Further reduction of the initial domain $X$.}

\begin{lemma}
  \label{lem:startingdomain}
We may assume without loss of generality that $0\leq x_2\leq \frac{1}{\operatorname{lcm}(v_1, v_2)}$.
\end{lemma}

\begin{proof}
We have fixed that $x_1=0$. But for every vector of starting positions, there are $v_1$ time rotations (i.e. advancing time by $1/v_1$) that will maintain this invariant. These correspond to considering every time when the runner $1$ crosses the origin. Consider the orbit of runner $2$ under the action of $\mathbb{Z}$ given by the times runner $1$ crosses the origin. The position of the first runner cycles around the track, advancing $v_2/v_1$ every time. This is, the same as advancing $\frac{v_2/\gcd(v_1,v_2)}{v_1/\gcd(v_1,v_2)}$. Therefore, every $v_1/\gcd(v_1,v_2)$ iterations the position of runner $2$ repeats. The orbit of runner $2$ in this action consists on $v_1/\gcd(v_1,v_2)$ equally spaced positions along the unit track. Then one of them must be in $[-\gcd(v_1,v_2)/v_1,0]$. The corresponding value $x_2$ is then in $[0,\frac{\gcd(v_1,v_2)}{v_1v_2}]$.  This is the bound that we claimed.
\end{proof}

We do not claim this is the best way to exploit the symmetry. A better result can be obtained by cascading this divisibility phenomenon across $\vv$, using time/space inversion, or, alternatively, using the domain generated by the Hermite row normal form (the diagonal entries form a fundamental domain even if they are not a lattice basis). 
But the improvement in Lemma~\ref{lem:startingdomain}  is enough for the low dimensional examples we deal with, especially if $v_1,v_2$ are coprime. 
The source repository has this improved domain implemented in \lstinline{scripts/bounds.py}

Note that the choice of order of coordinates changes the domain $X$. We have no reason right now to prefer one order over another, and we do not expect the gains from this choice to be significant.

\subsection{Our counterexamples and experimental results}
\label{subsec:counter}

Using the algorithm described in the previous sections we
have obtained several interesting counterexamples to the sLRC conjecture. These are all in the repository in a human-readable format in the folder \lstinline{/output} of \cite{code_git}. The repository contains all the velocity vectors we generated, categorized by $n$ and sorted increasingly by loneliness gap. It also has information on how to interpret and replicate these outputs, as well as the graphs we show below and animations for some selected velocity vectors.

\begin{proposition}
\label{prop:counters5}
\begin{enumerate}
\item $\sgap(1,2,3,4,5) =15/94 < 1/6$,
attained with the starting positions $\ss=\tfrac1{94}(0,46,38,47,72)$. 
\item Among the $\vv\in \Z_{>0}^5$ with $\sum v_i \leq 100$ there is no other vector  with $\sgap(\vv) < 1/6$ and there are $19$ with $\sgap(\vv) =\frac16$. The largest one is $(2, 3, 5, 8, 16)$, with $\sum_{i=1}^5 v_i=34$.
\end{enumerate}
\end{proposition}

Figure~\ref{fig:n5-s100} plots the values of $\sgap$ for all integer velocity vectors 
with $n=5$  and $\sum v_i \leq 100$. The figure strongly suggests that $(1,2,3,4,5)$ is most probably the only counterexample to sLRC for this $n$. 

\begin{figure}[htb]
\centerline{\includegraphics[width=.7\textwidth]{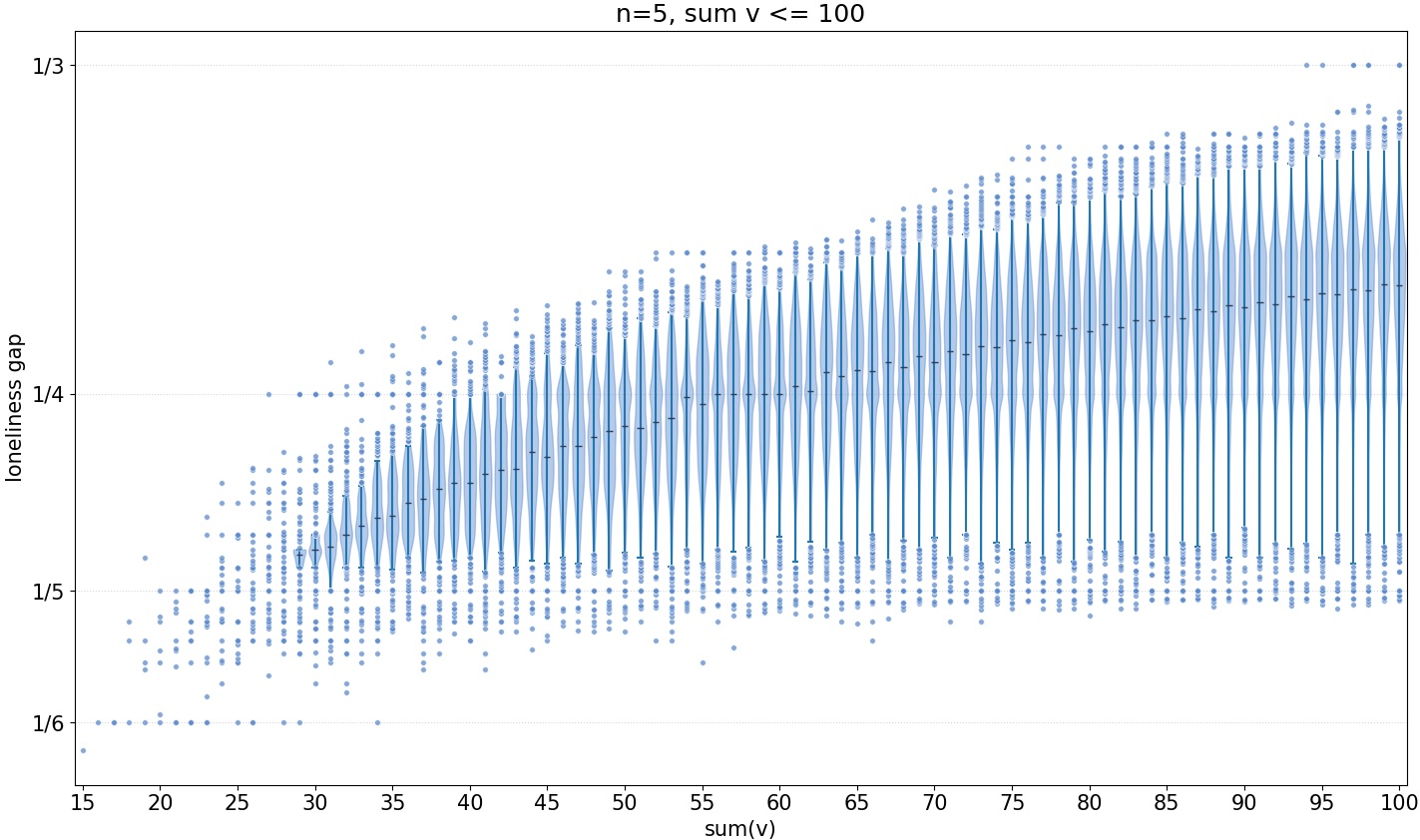}}
\caption{Values of $\sgap(\vv)$ for all integer velocity vectors $\vv = (v_1,\dots,v_5)$ with $\sum_{i} v_i \le100$. For each value of $\sum_i v_i$ the smallest $30$ and biggest $30$ values are drawn as dots, and the rest as a ``violin'' of relative width representing the amount of vectors with (approximately) that value of $\sgap(\vv)$.}
\label{fig:n5-s100}
\end{figure}


We then run our program with  $n=6, \sum\vv \leq 90$ and with $n=7, \sum\vv \leq 75$. 
The following statements summarize our findings, plotted in Figure~\ref{fig:n6-s90}.

\begin{proposition}
\label{prop:counters6}
\begin{enumerate}
\item $\sgap(2,3,4,5,6,8)\ =\ \frac{2}{15}\ <\ \frac17$, attained with the starting positions $\ss=\tfrac1{30}(0,29,17,0,16,22)$. \smallskip
\item $\sgap(1,2,3,4,5,6) = 9/67 $, attained with the starting positions $\ss=\tfrac1{67}(0, 39, 37, 51, 62, 54)$. 
\item Among the $493314$ sorted primitive vectors with $n=6$ and $\sum v_i \le 90$ there are:
\begin{itemize}
 \item $23$  with $\sgap(\vv)<1/7$ (including the two above). 
 The largest ones are $(1, 3, 4, 5, 7, 11)$ and  $(1, 2, 3, 4, 8, 13)$, with $\sum_i v_i =31$.
\item $21$ with $\sgap(\vv)=1/7$. The largest one is $(1, 3, 4, 5, 10, 30)$ with $\sum_i v_i =53$.
\end{itemize}
\end{enumerate}
\end{proposition}

\begin{figure}[htb]
\centerline{\includegraphics[width=.5\textwidth]{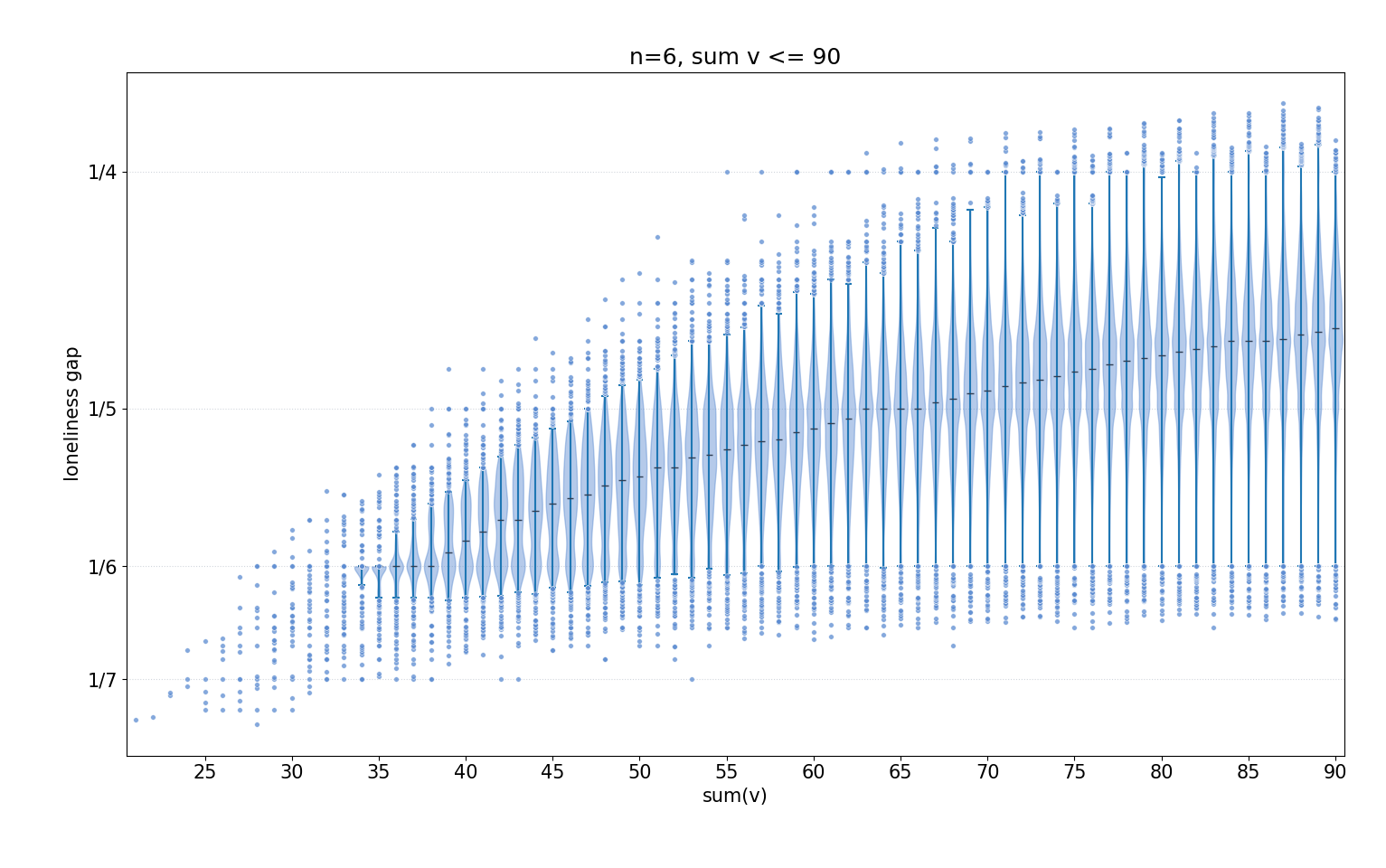}
\includegraphics[width=.5\textwidth]{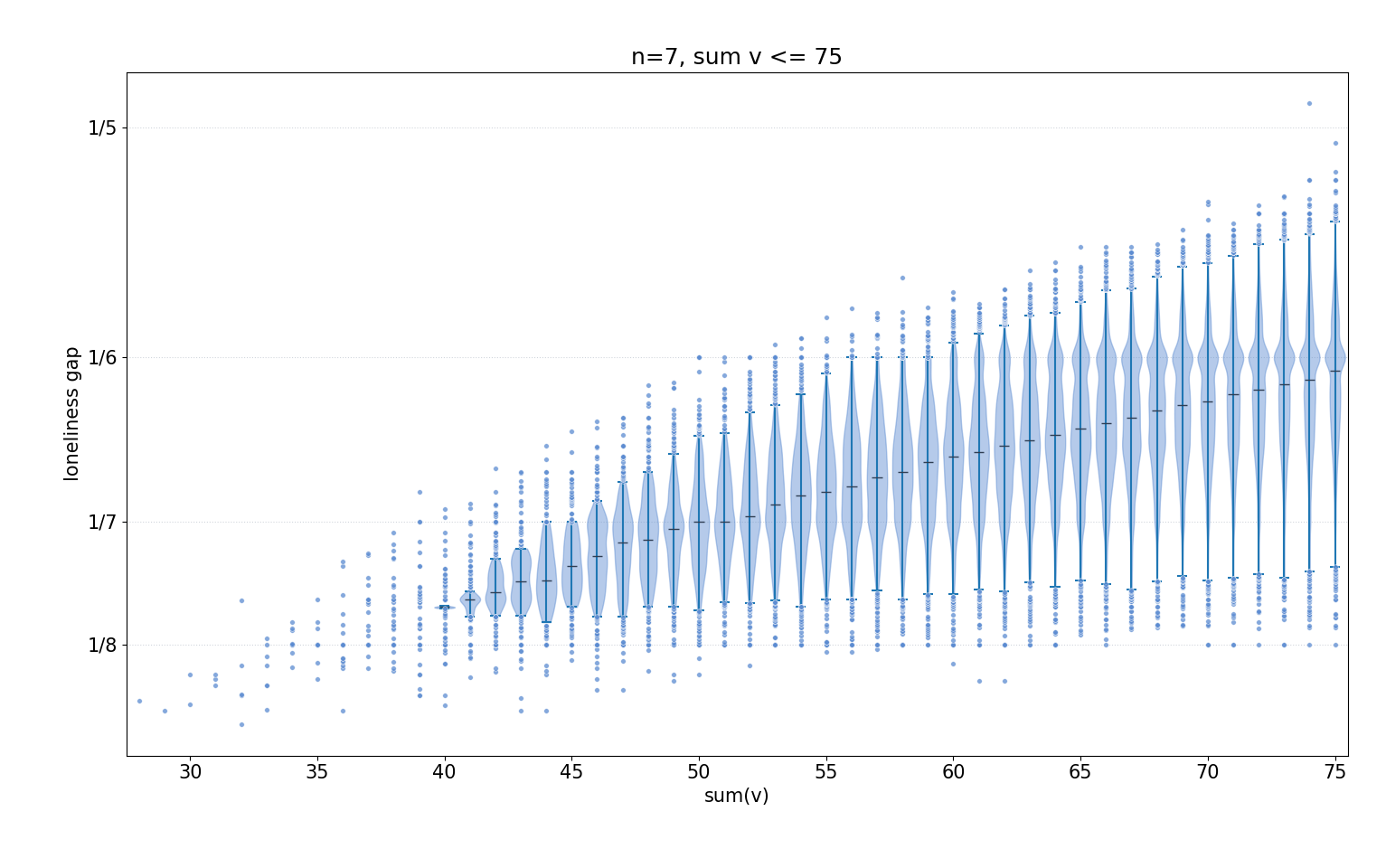}}
\caption{Same plots as in Figure~\ref{fig:n5-s100}, for $n=6,7$.}
\label{fig:n6-s90}
\label{fig:n7-s75}
\end{figure}

\begin{proposition}
\label{prop:counters7}
Among the 112501 sorted primitive vectors with $n=7$ and $\sum v_i \le 75$
there are  89  with $\sgap(\vv)<1/8$ and $269$ with $\sgap(\vv)=1/8$.
\end{proposition}

We believe that our list of counterexamples for $n=6$ is complete, but we are not so sure for $n=7$  because in this case we observe a fast increase in the number of counterexamples, and some of them have $\sum_i v_i$ not far from the limits of our computations (see again Figure~\ref{fig:n7-s75}). 
In fact, we believe that the number of counterexamples explodes for large $n$, with most counterexamples having small volume.
We have also checked that 
\begin{proposition}
\label{prop:countersn}
$\sgap(1,\dots,n) < 1/(n+1) $ for all $n\in [5,17]$.
\end{proposition}

For $n \leq 15$ we have been able to compute the exact value of $\sgap(1,\cdots, n)$,  shown in Table~\ref{table:consecutive}. For larger $n$ we can only certify an upper bound, but not the exact value. 
\begin{table}[htbp]
\setlength{\tabcolsep}{4pt}
\centering
\caption{Exact values of $\sgap(1,\dots,n)$ obtained in our computations.}
\label{table:consecutive}
\footnotesize
\begin{tabular}{cccccccc}
\toprule
  \!$n$\! & 2 & 3 & 4 & 5 & 6 & 7 & 8\\
\midrule
  $\sgap(1,\dots,n)$\quad\null & $\frac{1}{3}$ & $\frac{1}{4}$ & $\frac{1}{5}$ & $\frac{15}{94}$ & $\frac{9}{67}$ & $\frac{25}{214}$ & $\frac{1}{10}$ \\
\midrule
  $\sgap(1,\dots,n)^{-1}$ & ${3}$ & ${4}$ & ${5}$ & ${6.27}$ & ${7.44}$ & ${8.56}$ & ${10}$ \\
\bottomrule
\toprule 
  \!$n$\! & 9 & 10 & 11 & 12 & 13 & 14 & 15\\
\midrule
  $\sgap(1,\dots,n)$\quad\null & $\frac{5}{58}$ & $\frac{27}{349}$ & $\frac{1}{14}$ & $\frac{6}{91}$ & $\frac{1}{16}$ & $\frac{71}{1228}$ & $\frac{1}{18}$ \\
\midrule
  $\sgap(1,\dots,n)^{-1}$ & ${11.6}$ & ${12.93}$ & ${14}$ & ${15.17}$ & ${16}$ & ${17.30}$ & ${18}$ \\
\bottomrule
\end{tabular}
\end{table}
%

Observe that adding an extra runner with arbitrary velocity to a given velocity vector $\vv$ creates a longer vector $\vv'$ with 
$\sgap(\vv') \le \sgap(\vv)$. This explains the horizontal asymptotes in the lower part of the figures.
Moreover, since $\ss$ can be assumed rational, the new velocity can be chosen to make the new runner be at the origin at all times when $\sgap(\vv)$ was attained for $\vv$, implying 
$\sgap(\vv') < \sgap(\vv)$ (compare to the study of accumulation points in~\cite{girikravitz2023structurelonelyrunnerspectra}).
In particular, whenever a vector of size $n$ exists with $\sgap(\vv) \le \tfrac{1}{n+2}$, it produces infinitely many sLRC counterexamples for $n+1$.   Table~\ref{table:consecutive} then implies:

\begin{corollary}
\label{coro:infinite}
There are infinitely many counter-examples to sLRC of size $n$, for each $n\in \{9,\dots,16\}$.\qed
\end{corollary}

To quantify how bad counter-examples to sLRC are it is natural to introduce the parameter
\[
E_n :=  \tfrac1{\inf_{\vv\in \Z_{>0}^n} \sgap(\vv)}- (n+1),
\]
which is positive if and only if counter-examples to sLRC exist for that $n$.
After the first version of this paper was public Poliakova~\cite{Poliakova} has devised, with the help of an AI model, constructions that modify the standard vector $[n]$ to a \emph{shorter} vector $\vv'$ still having $\sgap(\vv') \le \tfrac{1}{n+1}$. Her constructions imply:

\begin{itemize}
\item \cite[Theorem 2]{Poliakova} For every $n\ge 95$, $E_n\ge 1$. In this construction the velocity vector consists of the even numbers in $[2, 2n+2]$ with the removal of $\{8,32,192\}$ and the insertion of $\{3,9\}$. Only eight runners are shifted.

\item \cite[Theorem 1]{Poliakova} For every $n\ge 1$, $E_n \ge \lfloor n/287\rfloor$. This construction iterates the previous one at different scales.
\end{itemize}
As in Corollary~\ref{coro:infinite}, $E_n\ge 1$ for $n\ge 95$ implies infinitely many counter-examples to sLRC for each $n\ge 96$.

\section*{Declarations}

\subsection*{AI statement}
AI models have been used only for minor editorial corrections and as a coding assistant. Details about the latter are in  \cite{code_git}.

\subsection*{Funding}
The three authors are supported by grant PID2022-137283NB-C21 funded by MICIU/AEI/10.13039/\allowbreak 501100011033 and FEDER, UE.
       
\subsubsection*{\bf Conflict of Interest}  The authors have no conflict of interest to declare.


\bibliographystyle{amsplain}
\bibliography{mybib}

@misc{Poliakova,
      title={More (shifted) runners, less loneliness}, 
      author={Daria Poliakova},
      year={2026},
      eprint={2609.23952},
      archivePrefix={arXiv},
      primaryClass={math.CO},
      url={https://arxiv.org/abs/2609.23952}, 
      note = {\url{https://arxiv.org/abs/2609.23952}}
}

@article{Schark,
 Author = {Schark, Reiner},
 Title = {Eine diophantische {A}pproximations-{F}unktion},
 FJournal = {Monatshefte f{\"u}r Mathematik},
 Journal = {Monatsh. Math.},
 Volume = {78},
 Pages = {131--146},
 Year = {1974},
 DOI={10.1007/BF01294776}
}

@book {Joswig-etc,
    AUTHOR = {Joswig, Michael},
     TITLE = {Essentials of tropical combinatorics},
    SERIES = {Graduate Studies in Mathematics},
    VOLUME = {219},
 PUBLISHER = {American Mathematical Society, Providence, RI},
      YEAR = {[2021] \copyright 2021},
     PAGES = {xx+398},
      ISBN = {978-1-4704-6653-4},
       DOI = {10.1090/gsm/219},
       URL = {https://doi-org.unican.idm.oclc.org/10.1090/gsm/219},
}

@article{ACS4slrz,
      title={Covering radii of $3$-zonotopes and the shifted Lonely Runner Conjecture}, 
      author={David Alcántara and Francisco Criado and Francisco Santos},
      eprint={2506.13379},
      archivePrefix={arXiv},
      primaryClass={math.CO},
      journal={Exp. Math.},
      doi={10.1080/10586458.2026.2651085},
      note = {in press (published online May 7, 2026)}
}

@misc{code_git,
      title={Shifted Lonely Runner Conjecture counterexample},
      author={Criado, Francisco},
      url={https://github.com/criado/shifted_lonely_runner_conjecture_counterexample},
      note = {\url{https://github.com/criado/shifted_lonely_runner_conjecture_counterexample}, code repository, 2026.}
}

@article{Rosenfeld1,
      title={The lonely runner conjecture holds for eight runners}, 
      author={Matthieu Rosenfeld},
      year={published electronically August 10, 2026},
      eprint={arXiv:2509.14111},
      Fjournal = {Mathematics of Computation},
      journal = {Math. Comp.},
      archivePrefix={arXiv},
      primaryClass={math.CO},
      doi={10.1090/mcom/4243}, 
      url = {https://doi.org/10.1090/mcom/4243},
      note = {\url{https://doi.org/10.1090/mcom/4243}}
}

@article{Trakulthongchai,
      title={Nine and ten lonely runners}, 
      author={Tanupat Trakulthongchai},
   JOURNAL = {Electron. J. Comb.},
  FJOURNAL = {Electronic Journal of Combinatorics},
    VOLUME = {33},
    issue = {2},
    pages = {\#2 P2.46},
      YEAR = {2026},
      eprint={2511.22427},
      archivePrefix={arXiv},
      primaryClass={math.CO},
      url={https://arxiv.org/abs/2511.22427}, 
      doi = {10.37236/14972},}

@misc{SungkawichaiTrakulthongchai,
      title={Eleven, twelve, and thirteen lonely runners}, 
      author={Touch Sungkawichai and Tanupat Trakulthongchai},
      year={2026},
      eprint={2604.23906},
      archivePrefix={arXiv},
      primaryClass={math.CO},
      url={https://arxiv.org/abs/2604.23906}, 
      note = {\url{https://arxiv.org/abs/2604.23906}}
}

@misc{Rosenfeld2,
      title={The lonely runner conjecture holds for nine runners}, 
      author={Matthieu Rosenfeld},
      year={2025},
      eprint={2512.01912},
      archivePrefix={arXiv},
      primaryClass={cs.DM},
      url={https://arxiv.org/abs/2512.01912}, 
      note = {\url{https://arxiv.org/abs/2512.01912}}
}

@article {Jarnik2,
    AUTHOR = {Jarn\'ik, Vojt\v{e}ch},
     TITLE = {Zwei {B}emerkungen zur {G}eometrie der {Z}ahlen},
   JOURNAL = {V\v estn\'ik Kr\'alovsk\'e{} \v Cesk\'e{} Spole\v cnosti Nauk.
              T\v r\'ida Matemat.-P\v r\'irodov\v ed.},
  FJOURNAL = {V\v estn\'ik Kr\'alovsk\'e{} \v Cesk\'e{} Spole\v cnosti Nauk.
              T\v r\'ida Matemat.-P\v r\'irodov\v ed.},
    VOLUME = {1941},
      YEAR = {1941},
     PAGES = {12}
}

@article {Jarnik1,
    AUTHOR = {Jarn\'ik, Vojt\v{e}ch},
     TITLE = {Remarque \'a l'article precedent de {M}. {M}ahler},
   JOURNAL = {\v Casopis P\v est. Mat. Fys.},
  FJOURNAL = {\v Casopis P\v est. Mat. Fys.},
    VOLUME = {68},
      YEAR = {1939},
     PAGES = {103--111}
}

@article{BeckSchymura,
    author = {Beck, Matthias and Schymura, Matthias},
    title = {Deep Lattice Points in Zonotopes, Lonely Runners, and Lonely Rabbits},
    journal = {Int. Math. Res. Not.},
    volume = {2024},
    number = {8},
    pages = {6553-6578},
    year = {2023},
    month = {10},
    issn = {1073-7928},
    doi = {10.1093/imrn/rnad232},
    url = {https://doi.org/10.1093/imrn/rnad232},
    eprint = {https://academic.oup.com/imrn/article-pdf/2024/8/6553/57274338/rnad232.pdf},
}

@book {BeckRobins,
    AUTHOR = {Beck, Matthias and Robins, Sinai},
     TITLE = {Computing the continuous discretely},
   EDITION = {Second},
      NOTE = {Integer-point enumeration in polyhedra},
 PUBLISHER = {Springer, New York},
      YEAR = {2015},
     PAGES = {xx+285},
      ISBN = {978-1-4939-2968-9; 978-1-4939-2969-6},
       DOI = {10.1007/978-1-4939-2969-6},
       URL = {https://doi.org/10.1007/978-1-4939-2969-6},
}

@book {DLRS2010triangulations,
    AUTHOR = {De Loera, Jes\'us A. and Rambau, J\"org and Santos, Francisco},
     TITLE = {Triangulations. {S}tructures for algorithms and applications},
    SERIES = {Algorithms and Computation in Mathematics},
    VOLUME = {25},
      NOTE = {},
 PUBLISHER = {Springer-Verlag, Berlin},
      YEAR = {2010},
     PAGES = {xiv+535},
      ISBN = {978-3-642-12970-4},
       DOI = {10.1007/978-3-642-12971-1},
       URL = {https://doi.org/10.1007/978-3-642-12971-1},
}

@article{ betkehenkwills1993successive,
 Author = {Betke, Ulrich and Henk, Martin and Wills, J{\"o}rg M.},
 Title = {Successive-minima-type inequalities},
 FJournal = {Discrete \& Computational Geometry},
 Journal = {Discrete Comput. Geom.},
 Volume = {9},
 Number = {2},
 Pages = {165--175},
 Year = {1993}
}

@book{ gruber2007convex,
	title = {Convex and Discrete Geometry},
	author = {Gruber, Peter M.},
	publisher = {Springer-Verlag},
	address = {Berlin},
	year = {2007},
	volume = {336},
	series = {Grundlehren der Mathematischen Wissenschaften},
	pages = {xiv+578}
}

@article{perarnauserra2024thelonely,
title = {The Lonely Runner Conjecture turns 60},
journal = {Comput. Sci. Rev.},
volume = {58},
pages = {100798},
year = {2025},
issn = {1574-0137},
doi = {https://doi.org/10.1016/j.cosrev.2025.100798},
url = {https://www.sciencedirect.com/science/article/pii/S1574013725000747},
author = {Guillem Perarnau and Oriol Serra}
}

@article{ girikravitz2023structurelonelyrunnerspectra,
title={The structure of Lonely Runner spectra}, 
DOI={10.1017/S0305004125101497}, 
journal={Math. Proc.. Camb. Philos. Soc.}, 
author={Vikram Giri and Noah Kravitz}, 
year={2025}, 
pages={1--19}}

@article{codenottisantosschymura2019the,
	author = {Codenotti, Giulia and Santos, Francisco and Schymura, Matthias},
	title = {The covering radius and a discrete surface area for non-hollow simplices},
   JOURNAL = {Discrete Comput. Geom.},
  FJOURNAL = {Discrete \& Computational Geometry. An International Journal of Mathematics and Computer Science},
	year = {2022},
	volume = {67},
	pages = {65--111},
	doi = {https://doi.org/10.1007/s00454-021-00330-3}
}

@Article{zonorunners,
 Author = {Henze, Matthias and Malikiosis, Romanos D.},
 Title = {On the covering radius of lattice zonotopes and its relation to view-obstructions and the lonely runner conjecture},
 FJournal = {Aequationes Mathematicae},
 Journal = {Aequationes Math.},
 ISSN = {0001-9054},
 Volume = {91},
 Number = {2},
 Pages = {331--352},
 Year = {2017},
 DOI = {10.1007/s00010-016-0458-3}
}

@misc{willsthesis,
      title={Zwei Probleme der inhomogenen diophantischen Approximation}, 
      author={Wills, J{\"o}rg M.},
      year={1965},
      note = {PhD Thesis, TU Berlin}
}

@article {wills65,
    AUTHOR = {Wills, J{\"o}rg M.},
     TITLE = {Widerlegung einer {A}ussage von {E}. {B}orel \"uber diophantische {A}pproximationen},
   JOURNAL = {Math. Zeitschr.},
    VOLUME = {89},
      YEAR = {1965},
     PAGES = {411--413}
}

@article {wills67,
    AUTHOR = {Wills, J{\"o}rg M.},
     TITLE = {Zwei {S}\"atze \"uber inhomogene diophantische {A}pproximation von
{I}rrationalzahlen},
   JOURNAL = {Monatsh. Math.},
    VOLUME = {71},
      YEAR = {1967},
     PAGES = {263--269}
}

@article {willslrc,
    AUTHOR = {Wills, J{\"o}rg M.},
     TITLE = {Zur simultanen homogenen diophantischen {A}pproximation. {I}},
   JOURNAL = {Monatsh. Math.},
    VOLUME = {72},
      YEAR = {1968},
     PAGES = {254--263}
}

@article {willslrc2,
    AUTHOR = {Wills, J{\"o}rg M.},
     TITLE = {Zur simultanen homogenen diophantischen {A}pproximation. {II}},
   JOURNAL = {Monatsh. Math.},
    VOLUME = {72},
      YEAR = {1968},
     PAGES = {368--381}
}

@article {cusicksimultaneous,
    AUTHOR = {Cusick, T. W.},
     TITLE = {Simultaneous diophantine approximation of rational numbers},
   JOURNAL = {Acta Arith.},
  FJOURNAL = {Polska Akademia Nauk. Instytut Matematyczny. Acta Arithmetica},
    VOLUME = {22},
      YEAR = {1972},
     PAGES = {1--9},
      ISSN = {0065-1036},
       DOI = {10.4064/aa-22-1-1-9},
       URL = {https://doi-org.unican.idm.oclc.org/10.4064/aa-22-1-1-9},
}

@article {cusickviewob,
    AUTHOR = {Cusick, Thomas W.},
     TITLE = {View-obstruction problems},
   JOURNAL = {Aequationes Math.},
  FJOURNAL = {Aequationes Mathematicae},
    VOLUME = {9},
      YEAR = {1973},
     PAGES = {165--170}
}

@article { 7lrc,
    AUTHOR = {Barajas, Javier and Serra, Oriol},
     TITLE = {The lonely runner with seven runners},
   JOURNAL = {Electron. J. Comb.},
  FJOURNAL = {Electronic Journal of Combinatorics},
    VOLUME = {15},
      YEAR = {2008},
     PAGES = {\#48, 18 pp. (electronic)}
}

@article {tao,
  author 				   = {Tao, Terence},
  title 				   = {Some remarks on the lonely runner conjecture},
  journal 				   = {Contrib. Discrete Math.},
  Year                     = {2018},
  Number                   = {2},
  Pages                    = {1--31},
  Volume                   = {13}
}

@article{sLRC,
    AUTHOR = {Beck, Matthias and Ho\c{s}ten, Serkan and Schymura, Matthias},
     TITLE = {Lonely Runner  Polyhedra},
   JOURNAL = {Integers},
    VOLUME = {19},
     PAGES = {\#A29, 13 pp},
      YEAR = {2019}
}

@article {KannanLovasz,
    AUTHOR = {Kannan, Ravi and Lov\'asz, L\'aszl\'o},
     TITLE = {Covering minima and lattice-point-free convex bodies},
   JOURNAL = {Ann. Math. (2)},
  FJOURNAL = {Annals of Mathematics. Second Series},
    VOLUME = {128},
      YEAR = {1988},
    NUMBER = {3},
     PAGES = {577--602},
       DOI = {10.2307/1971436},
       URL = {https://doi.org/10.2307/1971436}
}

@article {NamingLRC,
	AUTHOR = {Bienia, Wojciech and Goddyn, Luis and Gvozdjak, Pavol and
	Seb\H{o}, Andr\'as and Tarsi, Michael},
	TITLE = {Flows, view obstructions, and the lonely runner},
	JOURNAL = {J. Comb. Theory Ser. B},
	FJOURNAL = {Journal of Combinatorial Theory. Series B},
	VOLUME = {72},
	YEAR = {1998},
	NUMBER = {1},
	PAGES = {1--9},
	DOI = {10.1006/jctb.1997.1770},
	URL = {https://doi.org/10.1006/jctb.1997.1770}
}

@article {Schoenberg76,
	AUTHOR = {Schoenberg, Isaac J.},
	TITLE = {Extremum problems for the motions of a billiard ball. {II}.
	{T}he {$L\sb{\infty }$} norm},
	NOTE = {Nederl. Akad. Wetensch. Proc. Ser. A {\bf 79}},
	JOURNAL = {Indag. Math.},
	FJOURNAL = {},
	VOLUME = {38},
	YEAR = {1976},
	PAGES = {263--279}
}

@article{Malikiosis2024LinExpCheckingLRC,
    title = {Linearly-exponential checking is enough for the Lonely Runner Conjecture and some of its variants},
    author = {Romanos D.~Malikiosis and Francisco Santos and Matthias Schymura},
    volume={13}, 
    DOI={10.1017/fms.2025.10107}, 
    journal={Forum Math. Sigma}, 
    year = {2025},
    ISSN = {2050-5094},
    publisher = {Cambridge University Press},
    pages={e164}
}

@book{Erickson-algorithms,
  AUTHOR    = {Erickson, Jeff},
  TITLE     = {Algorithms},
  PUBLISHER = {Self-published},
  YEAR      = {2019},
  URL       = {https://jeffe.cs.illinois.edu/teaching/algorithms/#book},
  NOTE      = {Available at \url{https://jeffe.cs.illinois.edu/teaching/algorithms/\#book}},
}

@book{Halim-CP4,
  AUTHOR    = {Halim, Steven and Halim, Felix and Effendy, Suhendry},
  TITLE     = {Competitive Programming 4, Book 1},
  PUBLISHER = {Lulu},
  YEAR      = {2020},
}
\end{document}